\documentclass[leqno,10pt]{amsart}

\usepackage{amssymb}
\usepackage{amsmath}
\usepackage[usenames, dvipsnames]{color}
\usepackage{enumitem}
\usepackage{amsthm}
\usepackage{mathrsfs}
\usepackage{amsrefs}

\usepackage{marginnote}
\usepackage{todonotes}
\usepackage{hyperref}
\usepackage{esint}
\usepackage{verbatim}

\usepackage{etoolbox}

\DeclareMathOperator{\esssup}{ess\,sup}

\theoremstyle{plain}
\newtheorem{theorem}{Theorem}[section]
\newtheorem{lemma}[theorem]{Lemma}

\theoremstyle{definition}
\newtheorem{definition}[theorem]{Definition}

\theoremstyle{remark}
\newtheorem{remark}[theorem]{Remark}

\newcommand{\R}{\mathbb{R}}
\newcommand{\N}{\mathbb{N}}

\newcommand{\p}{\partial}

\newcommand{\bu}{\boldsymbol{u}}
\newcommand{\bx}{\boldsymbol{x}}
\newcommand{\bv}{\boldsymbol{v}}
\newcommand{\bw}{\boldsymbol{w}}
\newcommand{\bphi}{\boldsymbol{\varphi}}
\newcommand{\bF}{\boldsymbol{F}}
\newcommand{\bg}{\boldsymbol{g}}
\newcommand{\bPhi}{\boldsymbol{\phi}}
\newcommand{\bpsi}{\boldsymbol{\psi}}
\newcommand{\ltwo}[1]{{L^2\left(#1\right)}}
\newcommand{\lfour}[1]{{L^4\left(#1\right)}}
\newcommand{\linf}[1]{{L^\infty\left(#1\right)}}
\newcommand{\hdot}[1]{{\dot H^1\left(#1\right)}}
\newcommand{\hdotzig}[1]{{\dot H^1_{0,\sigma}\left(#1\right)}}

\newcommand{\upper}{{\R^2_+}}
\newcommand{\balln}{{B^+_n}}
\newcommand{\ballm}{{B^+_m}}
\newcommand{\com}[1]{{C^\infty_c\left(#1\right)}}
\newcommand{\comsig}[1]{{C^\infty_{c,\sigma}\left(#1\right)}}

\def\XXint#1#2#3{{\setbox0=\hbox{$#1{#2#3}{\int}$ }
\vcenter{\hbox{$#2#3$ }}\kern-.6\wd0}}

\title[N-S equations in the upper-half plane]{On stationary Navier-Stokes equations in the upper-half plane}

\author[A. Calderon]{Adrian D. Calderon}
\address[A. Calderon]{Department of Department of Mathematics and Statistics, Boston University, 665 Commonwealth Ave,
Boston, MA 02215}
\email{acaldero@bu.edu}

\author[V. Le]{Van Le}
\address[V. Le]{Department of Mathematics, University of Tennessee, 227 Ayres Hall,
1403 Circle Drive, Knoxville, TN 37996-1320 }
\email{nle12@vols.utk.edu}

\author[T. Phan]{Tuoc Phan}
\address[T. Phan]{Department of Mathematics, University of Tennessee, 227 Ayres Hall,
1403 Circle Drive, Knoxville, TN 37996-1320}
\email{phan@utk.edu}

\subjclass[2020]{76D05, 76D03, 35Q30}
\keywords{Stationary Navier-Stokes equations, Existence, Uniqueness}

\begin{document}

\begin{abstract} We study  the incompressible stationary Navier-Stokes equations in the upper-half plane with homogeneous Dirichlet boundary condition and non-zero external forcing terms. Existence of weak solutions is proved under a suitable condition on the external forces. Weak-strong uniqueness criteria based on  various growth conditions at the infinity of weak solutions are also given. This is done by employing an energy estimate and a Hardy's inequality. Several estimates of stream functions are carried out and two density lemmas with suitable weights for the homogeneous Sobolev space on 2-dimensional space are proved.
\end{abstract}
\maketitle
\section{Introduction and main results}

Let $\upper:=\left\{ \bx = (x_1, x_2) \in\R^2:x_2>0\right\}$ be the upper-half plane. We consider the following stationary Navier-Stokes equations of the incompressible fluid in $\upper$ 
\begin{equation} \label{1.1}
\left\{
\begin{array}{cccl}
   - \boldsymbol\Delta\bu + (\bu\cdot\nabla) \bu + \nabla P  + \text{div}(\bF) & = &  0 & \quad \text{in} \quad \upper\\
    \nabla\cdot\bu & = & 0 & \quad \text{in} \quad \upper\\
   \qquad \bu  & = & 0 & \quad \text{on} \quad \p\upper,
   \end{array} \right.
 \end{equation}
where $\bu=(u_1,u_2):\upper\rightarrow\R^2$ is the unknown velocity of the fluid, $P$ is the unknown fluid pressure, $\bF: \upper\rightarrow\R^{2\times 2}$ is a given forcing term. We are interested in studying the existence and uniqueness of weak solutions to \eqref{1.1}. One of the reasons for our interest in studying the problem in $\R^2_+$ is based on the unknown asymptotic behavior as $|\bx| \rightarrow \infty$ of weak solutions in the homogeneous Sobolev space $\hdotzig{\R^2_+}$ to \eqref{1.1}. See the paragraph after Theorem \ref{theorem2.1} below for more discussion.

\smallskip
To introduce our main results, let us provide several definitions and  notations used in the paper. For a given open, non-empty set  $\mathcal{D} \subseteq \R^2$, we denote $C^\infty_c(\mathcal{D})$ the set of all compactly supported smooth functions on $\mathcal{D}$, and
\begin{align*}
    \comsig{\mathcal{D}} =\{\bphi\in\com{\mathcal{D}}^2:\nabla\cdot\bphi=0\}.
\end{align*} 
The homogeneous Sobolev space $\hdot{\mathcal{D}}$ is defined by
\begin{align*}
    \hdot{\mathcal{D}} =\{u\in L^1_{\textup{loc}}(\mathcal{D}):\nabla u\in\ltwo{\mathcal{D}}\},
\end{align*}
and it is endowed with the semi-norm 
\[
    \|u\|_\hdot{\mathcal{D}}=\|\nabla u\|_\ltwo{\mathcal{D}}.
\]
Next, we denote
\[ 
\hdotzig{\mathcal{D}}\  \textnormal{the completion of } \comsig{\mathcal{D}}\textnormal{ in  }\hdot{\mathcal{D}} \times \hdot{\mathcal{D}}.
\]
Also, recall that for a given weight $\mu$ defined on $\mathcal{D}$, and for $q \in [1, \infty]$, the weighted Lebesgue space $L^q(\mathcal{D},\mu)$ is defined as
\[
L^q(\mathcal{D},\mu) = \{f : \mathcal{D} \rightarrow \R: \|f\|_{L^q(\mathcal{D}, \mu)} <\infty \}
\]
where
\[
\|f\|_{L^q(\mathcal{D},\mu)} = \left \{
\begin{array}{ll}
\displaystyle{\Big (\int_{\mathcal{D}} |f(\bx)|^q \mu(\bx) d\bx }\Big)^{1/q} & \quad \text{if} \quad q < \infty,\\
\esssup_{\mathcal{D}}|f(\bx) \mu(\bx)| & \quad \text{if} \quad q =\infty.
\end{array} \right.
\]
We also recall the definition of weak solutions to \eqref{1.1}. 
\begin{definition} Let $\bF=(\bF_1,\bF_2) \in\ltwo{\R^2_+}^{2\times2}$, with $\bF_1 = (F_{11}, F_{12})$ and $\bF_2 = (F_{21}, F_{22})$. A vector field $\bu=(u_1,u_2)\in\hdotzig{\R^2_+}$ is said to be a weak solution of \eqref{1.1} 
if
\begin{equation} \label{1.2}
    \sum_{i=1}^2\int_{\R_+^2}\big(\nabla u_i\cdot\nabla\varphi_i+(\bu\cdot\nabla)u_i\varphi_i\big)\,d\bx= \sum_{i=1}^2\int_{\R_+^2}\bF_i\cdot\nabla\varphi_i\,d\bx,
\end{equation}
for all $\bphi=(\varphi_1,\varphi_2)\in C^\infty_{c\\
	,\sigma}(\R^2_+)$.
\end{definition}
Throughout the paper, we denote the following weights
\begin{equation} \label{weight-def}
\omega_0(\bx) =\frac{|x_2|^2 |\bx|}{|x_2| + 4|\bx|},  \quad \omega_1(\bx) = |x_2| |\bx|, \quad \text{and} \quad \omega_2(\bx) = |x_2|^2/4,
\end{equation}
for $\bx= (x_1, x_2) \in \R^2_+$. Observe that
\begin{equation} \label{weight-relation}
\frac{1}{\omega_1(\bx)} + \frac{1}{\omega_2(\bx)} = \frac{1}{\omega_0(\bx)}.
\end{equation}
We point out that  $\|\cdot \|_\hdot{\R^2_+}$ is a norm due to the Hardy inequality
\begin{equation} \label{hardy-1}
\|u\|_{L^2(\R^2_+, \omega^{-1})} \leq 4 \|\nabla u\|_{L^2(\R^2_+)}, \quad \forall u \in C_c^\infty({\R^2_+})
\end{equation}
for $\omega \in \{\omega_0, \omega_1, \omega_2\}$, see Lemma \ref{Hardy} below.

The first result of the paper is the existence of weak solutions $\bu = (u_1, u_2) \in\hdotzig{\upper} \cap L^2(\R^2_+, \omega^{-1})$ to the stationary Navier-Stokes equations \eqref{1.1} on the upper-half plane $\upper$ with $\omega \in \{\omega_0, \omega_1, \omega_2\}$. 
\begin{theorem} \label{theorem2.1} For every $\bF\in L^2(\R^2_+)^{2\times2}$, there exists a weak solution $\bu = (u_1, u_2) \in\hdotzig{\upper}$ to the stationary Navier-Stokes equations \eqref{1.1} satisfying
\begin{equation} \label{2.1}
    \|\nabla\bu\|_{L^2(\R^2_+)}^2 \leq \langle\nabla\bu, \bF\rangle_\ltwo{\R^2_+} 
\end{equation}
and
\begin{equation} \label{est-0321}
\|\nabla\bu\|_{L^2(\R^2_+)} \leq \|\bF\|_{L^2(\R^2_+)}, \quad \|\bu\|_{L^2(\R^2_+, \omega^{-1})} \leq 4\|\bF\|_{L^2(\R^2_+)},
\end{equation}
for any $\omega \in \{\omega_0, \omega_1, \omega_2\}$. 
\end{theorem}

We would like to point out that though the Hardy type inequality \eqref{hardy-1} is available, it is not sufficient  to characterize the behavior at infinity for functions in $\hdotzig{\upper}$. In fact, it is known that $|\bu(\bx)| \rightarrow 0$ as $|\bx| \rightarrow \infty$ in the sense defined in Remark \ref{zero-infity} below, for $\bu \in \hdotzig{\upper}$. However, the rate of the convergence is not clear. Moreover, note that there is no Sobolev embedding theorem from $\hdotzig{\upper}$ into a Lebesgue space as we are in an unbounded two dimensional domain.  As such, the uniqueness of weak solutions in $\hdotzig{\upper}$ to \eqref{1.1} is a challenging problem. Specifically, if $\bu, \bar{\bu} \in \hdotzig{\upper}$ are two weak solutions to \eqref{1.1}, then $\bw : = \bu -\bar{\bu} \in \hdotzig{\upper}$ is a weak solution to
\begin{equation} \label{diff-show}
\left\{
\begin{array}{cccl}
-\boldsymbol\Delta \bw + \bar{\bu} \cdot \nabla \bw + \bw \cdot \nabla \bu + \nabla \bar{P} & = & 0 & \quad \text{in} \quad \R^2_+, \\
\nabla \cdot \bw & = & 0 & \quad \text{in} \quad \R^2_+,\\
\bw & = & 0 & \quad \text{on} \quad \partial \R^2_+.
\end{array} \right.
\end{equation}
As $\nabla \bu \in L^2(\R^2_+)$ and not much is known on the integrability of $\bw$, it is not easy to control the term $\bw \cdot \nabla \bu$. Similar issue also appears in controlling the term $\bar{\bu} \cdot \nabla \bw$ when applying the integration by parts. As such, Liouville type theorem for solutions in $\hdotzig{\upper}$ to \eqref{diff-show} is not well understood. The same issue also appears when $\R^2_+$ is replaced by $\R^2$. Interested readers can find in \cite{Amick, Fin, GSohr, Lad, galdi, Naz, Smith} for  some classical results on the asymptotic behavior as $|\bx| \rightarrow \infty$, and on the uniqueness of solutions to the Navier-Stokes equations in various types of domains. Similarly, see also \cite{GKR, guillod, KP, G-W, Yamazaki} and the references therein for some recent developments.   

\smallskip
In this work, we impose certain extra assumptions on the behavior at infinity for solutions to prove the weak-strong uniqueness of solutions.  We denote  
\begin{equation} \label{X-def}
\mathcal{X}  = \Big\{\bv \in \hdotzig{\upper}:  \int_{0}^\infty x_2\| \bv(\cdot, x_2)\|_{L^\infty(\R)}^2 dx_2 <\infty \Big\}.
\end{equation}
Then, we have the following results on weak-strong uniqueness of the weak solutions obtained in Theorem \ref{theorem2.1}, which is also the main result of the paper.
\begin{theorem} \label{theorem2.2} Let $\bF\in\ltwo{\upper}^{2\times2}$ and assume that there is a weak solution $\bar{\bu} \in \hdotzig{\upper}$ to the  Navier-Stokes equations \eqref{1.1} which satisfies \eqref{2.1} and one of the following conditions
\begin{itemize}
\item[\textup{(i)}] There is $i =0,1,2$ so that $\bar{\bu} \in L^\infty(\R^2_+, \omega^{1/2}_i)$ and
\begin{equation} \label{2.2}
    \|\bar{\bu}\|_{L^\infty(\R^2_+, \omega_i^{1/2})}< 1/4.
\end{equation}
\item[\textup{(ii)}] The weak solution $\bar{\bu} \in \mathcal{X}$  and 
\begin{equation}  \label{2.2-2}
\left(\int_0^\infty x_2\|\bar{\bu}(\cdot, x_2)\|_{L^\infty(\R)}^2 dx_2  \right)^{1/2} < 1.
\end{equation}
\end{itemize}
Then for every weak solution $\bu \in \hdotzig{\upper}$ to \eqref{1.1} satisfying \eqref{2.1}, we have $\bu=\bar{\bu}$.
\end{theorem}

We would like to point out that (i) in Theorem \ref{theorem2.2} improves a result in \cite{G-W}. Precisely, in \cite[Theorem 6.1]{G-W}, the same weak-strong uniqueness assertion was proved under the assumption that $\bar{\bu} \in L^\infty(\R^2_+, \omega_2^{1/2})$, $|\nabla \bar{\bu}| \in L^\infty(\R^2_+, \omega_2)$,  and  $\|\bar{\bu}\|_{L^\infty(\R^2_+, \omega_2^{1/2})}$ is sufficiently small (the smallness was not quantified). See also \cite{guillod} for similar results on the weak - strong uniqueness of weak solutions to \eqref{1.1} in $\R^2$ under an assumption that a weighted $L^\infty$-norm of $\bar{\bu}$ is sufficiently small with a suitable weight.  Here, (i) of Theorem \ref{theorem2.2} does not require $|\nabla \bar{\bu}| \in L^\infty(\R^2_+, x_2^2)$.  Moreover, besides providing a sufficient condition using several different weights, (i) explicitly quantifies a  required condition on the size of $\|\bar{\bu}\|_{L^\infty(\R^2_+, \omega_i^{1/2})}$, namely $\|\bar{\bu}\|_{L^\infty(\R^2_+, \omega_i^{1/2})}<1/4$ for either one of $i =0,1,2$. Certainly, it is not clear if this quantification is optimal. Note also that if $\|\bar{\bu}\|_{L^\infty(\R^2_+, \omega_i^{1/2})}<1/4$ holds for either $i=1$ or $i=2$, then it holds for $i =0$ due to \eqref{weight-relation}. But, the reverse direction is unclear.  It is also unclear that if $\|\bar{\bu}\|_{L^\infty(\R^2_+, \omega_2^{1/2})}<1/4$ then $\|\bar{\bu}\|_{L^\infty(\R^2_+, \omega_1^{1/2})}<1/4$ due to the anisotropic structure of $\omega_1$. Regarding this, we also point out that  the anisotropic  type mixed-norm condition (ii) in Theorem \ref{theorem2.2} does not seem to appear elsewhere. Note that  both conditions \eqref{2.2} - \eqref{2.2-2} are invariant under the natural scaling for \eqref{1.1} as pointed out in  Remark \ref{remark-scaling} below.

\smallskip
The literature studying \eqref{1.1} in different types of domains in two or higher dimensions is vast. Let us briefly discuss a few research lines related to Theorems \ref{theorem2.1}-\ref{theorem2.2}. Since the well-known work \cite{Leray, Leray-2}, the uniqueness of solutions to Navier-Stokes equations \eqref{1.1} is a classical open problem for both stationary and non-stationary equations.  As we already discussed, for the stationary equations in an unbounded  domain $\mathcal{D} \subset \R^2$, the behavior of solutions at infinity is not well-understood. As such the uniqueness problem is interesting and challenging. Special attention has been paid to the problem in two-dimensional aperture domain with non-zero flux introduced and studied in \cite{Heywood}. It was shown in \cite{Gal-14} that the velocity tends to zero in the $L^2$-norm for arbitrary values of the flux. For small fluxes and with the symmetric assumption with respect to the $x_2$-axis, it was proved in \cite{Gal-14, Naz} that the asymptotic behavior is given by a Jeffery-Hamel solution. The problem on the asymptotic
behavior of the non-symmetric solutions in the two-dimensional aperture domains still remains open. In another research direction, a problem in a straight channel connected to a half-plane  was also studied in  \cite{Naz-1, Naz-2}, and under certain conditions, the asymptotic behavior given by a Jeffery-Hamel flow in the half-plane and by the Poiseuille flow in the channel was proved.  Similarly, problem  \eqref{1.1} in the whole plane and exterior domains in two or higher dimensions has been studied extensively. One can find in  \cite{GKR, GSohr, guillod, KP, G-W, Yamazaki} and the references therein for additional developments on existence and uniqueness of solutions. See also the books \cite{Lad, galdi} and classical work \cite{Amick, Fin, GSohr,  Heywood, gilbarg, Smith} for more classical results on uniqueness, asymptotic behaviors for solutions of both time-dependent and time-independent Navier-Stokes equations \eqref{1.1}. In a special interest, it  is worth pointing out that in the case where there is no external force and when the domain is $\R^2$, Liouville type theorem was proved in \cite{GW}. Moreover,  when the external force has compact support,  the problem \eqref{1.1} with a prescribed constant limit of a solution at spatial infinity in the whole plane $\R^2$ was also studied in \cite{GKR}. It was shown in this paper that if the force is small enough, then the solution to this problem is unique.  See also \cite{DBC} for a recent interesting result on non-uniqueness of weak solutions to \eqref{1.1} in bounded two dimensional domains.

\begin{remark} \label{remark-scaling} The conditions \eqref{2.2} - \eqref{2.2-2} are invariant under the natural scaling of the Navier-Stokes equations. In particular, for any $\lambda >0$ and let $\bar{\bu}_\lambda(\bx) = \lambda\bar{\bu}(\lambda \bx)$. We have
\[
\begin{split}
\|\bar{\bu}_{\lambda}\|_{L^\infty(\R^2_+, \omega_i^{1/2})} & = \displaystyle{\esssup_{\bx \in \R^2_+} \omega_i^{1/2}(\bx) |\bar{\bu}_\lambda(\bx)|} \\
& = \lambda \displaystyle{\esssup_{\bx \in \R^2_+} \omega_i^{1/2}(\bx) |\bar{\bu} (\lambda \bx)|} \\
& = \displaystyle{\esssup_{\bx \in \R^2_+} \omega_i^{1/2}(\lambda \bx) |\bar{\bu} (\lambda \bx)|}= \|\bar{\bu}\|_{L^\infty(\R^2_+, \omega_i^{1/2})},
\end{split}
\]
where we used the fact that $\lambda \omega_i^{1/2}(\bx) = \omega_i^{1/2}(\lambda \bx)$ for all $\bx \in \R^2_+$. Hence \eqref{2.2} is invariant under the above scaling. Similarly, by using a change of variables, we also have
\begin{align*}
    \left(\int_0^\infty x_2\|\bar{\bu}_\lambda(\cdot, x_2)\|_{L^\infty(\R)}^2 dx_2  \right)^{1/2}& = \left(\int_0^\infty x_2\big[\displaystyle{\esssup_{x_1 \in \R}}|\bar{\bu}_\lambda(x_1, x_2)|\big]^2 dx_2  \right)^{1/2}\\
    & =\left(\int_0^\infty x_2\big[\displaystyle{\esssup_{x_1 \in \R}}|\lambda\bar{\bu}(\lambda x_1,\lambda x_2)|\big]^2 dx_2  \right)^{1/2}\\
    & = \left(\int_0^\infty \lambda^2 x_2\big[\displaystyle{\esssup_{x_1 \in \R}}|\bar{\bu}(\lambda x_1,\lambda x_2)|\big]^2 dx_2  \right)^{1/2}\\
    & = \left(\int_0^\infty y_2\big[\displaystyle{\esssup_{x_1 \in \R}}|\bar{\bu}(\lambda x_1, y_2)|\big]^2 dy_2  \right)^{1/2}\\
    & = \left(\int_0^\infty y_2\|\bar{\bu}(\cdot, y_2)\|_{L^\infty(\R)}^2 dy_2  \right)^{1/2}.
\end{align*}
Therefore, \eqref{2.2-2} is invariant under the scaling $\bar{\bu} \rightarrow \bar{\bu}_\lambda$. 
\end{remark}

\smallskip

\smallskip
The rest of the paper is organized as following.  In Section \ref{Prelim-sec}, we recall and prove several preliminary estimates needed to prove our main results. In particular, a Hardy type inequality, several estimates on stream functions are established, and some results on the density of classes of functions in $\hdotzig{\upper}$ are proved.   Then, in the last section,  Section \ref{existence-sec}, we give the proofs of Theorem \ref{theorem2.1} and Theorem \ref{theorem2.2}.
 
\section{Preliminary inequalities and estimates} \label{Prelim-sec}

This section provides several inequalities and estimates needed for proving Theorem \ref{theorem2.1} and Theorem \ref{theorem2.2}.  
\subsection{Hardy inequality and trace inequality}
We begin with the following lemma on the Hardy inequality.
\begin{lemma}[Hardy inequality] \label{Hardy} For every $v\in \dot{H}^1_0(\R^2_+)$, it holds that
\begin{equation} \label{3.3}
    \|v\|_{L^2({\upper}, \omega^{-1}_i)} \leq 4 \| \nabla v\|_\ltwo{\upper},
\end{equation}
where $\omega_i$ is defined in \eqref{weight-def} for $i =0,1,2$.
\end{lemma}
\begin{proof}  As $C_c^\infty(\R^2_+)$ is dense in $\dot{H}^1_0(\R^2_+)$, we may assume without loss of generality that $v \in C_c^\infty(\R^2_+)$. Recall that by the standard Hardy's inequality, see  \cite[Theorem 1, p. 214]{mazya} for example, we have
\begin{equation} \label{3.4}
    \|u\|_{L^2{(\R^2_+, |\bx|^{-1})}}\leq4\| \nabla u\|_{L^2(\R^2, x_2)}
\end{equation} 
for all $u\in C^\infty_c(\overline{\R^2_+})$.  Then, for the given $v\in\com{\upper}$, we apply \eqref{3.4} with 
\[ u(\bx):=x_2^{-1/2}v(\bx) \quad \text{for} \quad \bx = (x_1, x_2) \in \R^2_+ \] 
to obtain
\begin{align*}
    \|v\|_{L^2(\upper, x_2^{-1}|\bx|^{-1})}\leq4\left(\int_\upper \big( \frac{1}{4}x_2^{-2}v^2-x_2^{-1}vv_{x_2}+|\nabla v|^2\big)\,d\bx\right)^{1/2}.
\end{align*}
We then integrate by parts on the second term on the right hand side of the inequality to infer
\begin{align*}
      \|v\|_{L^2(\R_+^2, x_2^{-1}|\bx|^{-1})}\leq4\left(\int_\upper \big(-\frac{1}{4}x_2^{-2}v^2+|\nabla v|^2\big)\,d\bx\right)^{1/2}.
\end{align*}
This implies
\begin{align*}
    \|v \|_{L^2({\upper}, x_2^{-1}|\bx|^{-1})}^2 + 4\|v\|^2_{L^2({\upper}, x_2^{-2})} \leq16\|\nabla v\|^2_\ltwo{\upper},
\end{align*}
which proves the desired result.
\end{proof}
Due to Lemma \ref{Hardy}, we have the following remark.
\begin{remark} \label{zero-infity} If $\bu \in \hdotzig{\upper}$, then $\bu (\bx) =0$ as $|\bx| \rightarrow \infty$ in the sense that
\[
\lim_{r \rightarrow \infty} \int_{\beta_1}^{\beta_2} |\bu(r \cos(\theta), r \sin(\theta))|^2 d\theta =0
\]
for every $\beta_1, \beta_2 \in [0,\pi]$ and $\beta_1 < \beta_2$.
\end{remark}
\begin{proof} Let us denote
\[
\Omega_{r} = \Big \{ \bx= (s\cos(\theta), s\sin(\theta)): r < s < 2r, \quad \beta_1 < \theta < \beta_2\Big\},
\]
and
\[
S_r = \Big \{ \bx= (r\cos(\theta), r\sin(\theta)):  \beta_1 < \theta < \beta_2\Big\}.
\]
By applying the trace inequality on $\Omega_1$ and then using the scaling argument, we can find a constant $N >0$ such that
\[
\frac{1}{r}\|\bu\|_{L^2(S_r)}^2 \leq N \left( \frac{1}{r^2}\|\bu\|_{L^2(\Omega_r)}^2 + \|\nabla \bu\|_{L^2(\Omega_r)}^2\right), \quad \forall \ r \in (0, \infty).
\]
We note that
\[
\frac{1}{r}\|\bu\|_{L^2(S_r)}^2 = 2\pi \int_{\beta_1}^{\beta_2} |\bu(r\cos(\theta), r\sin(\theta))|^2 d\theta
\]
and
\[ x_2 \leq 2r  \quad  \forall \ \bx = (x_1, x_2) \in \Omega_r.
\]
Then, \[
\begin{split}
\int_{\beta_1}^{\beta_2} |\bu(r\cos(\theta), r\sin(\theta)|^2 d\theta & \leq N \left(4\|\bu\|_{L^2(\Omega_r, x_2^{-2})}^2 + \|\nabla \bu\|_{L^2(\Omega_r)}^2\right) \\
& \rightarrow 0 \quad \text{as} \quad r \rightarrow \infty,
\end{split}
\]
where we have used Lemma \ref{Hardy} in the last step. The proof is then completed.
\end{proof}
The next lemma gives two simple mixed-norm trace estimates of functions in homogeneous Sobolev space $\dot{H}^1_0(\R^2_+)$. 
\begin{lemma}[Trace type inequality]\label{Hardy-2}  There is  a constant $N>0$ such that 
\begin{equation} \label{L2-01-v}
\displaystyle{\esssup_{x_1 \in \mathbb{R}}\| v(x_1, \cdot)\|_{L^2((0,1))}} \leq N \|\nabla v\|_{L^2(\mathbb{R}^2_+)},
\end{equation}
and
\begin{equation} \label{3.3-2}
\displaystyle{ \esssup_{x_2 \in (0,\infty)} \big(x_2^{-1/2}\|v(\cdot, x_2)\|_{L^2(\R)}\big)}\leq   \| \nabla v\|_\ltwo{\upper},
\end{equation}
for $ v \in \dot{H}^1_0(\R^2_+)$.
\end{lemma}
\begin{proof}  Due to the density of $C_c^\infty(\R^2_+)$ in $\dot{H}^1_0(\R^2_+)$, we may assume without loss of generality that $v \in C_c^\infty(\R^2_+)$. As $v(\cdot, 0) \equiv 0$, by applying the trace inequality and Poincar\'e inequality on $(0,1) \times (0,1)$, and then using a translation in $x_1$-variable, we find a constant $N>0$ such that
\[
\begin{split}
\|v(x_1, \cdot)\|_{L^2((0,1))} & \leq N \Big( \|v\|_{L^2((x_1 -1, x_1) \times (0,1))} + \|\nabla v\|_{L^2((x_1 -1, x_1) \times (0,1))} \Big)  \\
& \leq N \| \nabla v\|_{L^2((x_1 -1, x_1) \times (0,1))},
\end{split}
\]
for any $x_1 \in \R$. Therefore,
\[
\|v(x_1, \cdot)\|_{L^2((0,1))} \leq N \| \nabla v\|_{L^2(\R^2_+)}, \quad  \ \forall \ x_1 \in \R
\]
and the assertion \eqref{L2-01-v} follows.

\smallskip
We now prove \eqref{3.3-2}. By the fundamental theorem of calculus and the fact that $ v(x_1, 0) =0$, we have
\[
v(x_1, x_2)  = \int_0^{x_2} \partial_2 v(x_1, s) ds.
\]
Therefore, it follows from H\"{o}lder inequality that
\[
|v(x_1, x_2)|^2 \leq x_2 \int_0^{x_2} |\partial_2 v(x_1, s) |^2 ds  \leq x_2  \int_0^{\infty} |\partial_2 v(x_1, s) |^2 ds
\]
and hence
\[
\int_{\R}|v(x_1, x_2)|^2 dx_1 \leq x_2 \|\nabla v\|_{L^2(\R^2_+)}^2
\]
and \eqref{3.3-2} follows. The proof of the lemma is completed.
\end{proof}
\subsection{Estimates of stream functions} Let us recall that for each $\bv = (v_1, v_2) \in\hdotzig{\upper}$, the stream function $\psi$ associated with $\bv$ is defined by
\begin{equation} \label{stream-def}
\psi(\bx)= \int_{\Gamma} (v_2\,dx - v_1 \,dy),
\end{equation}
where $\Gamma$ is a continuous curve connecting $(0,0)$ and $\bx= (x_1,x_2)$.  Because $\nabla\cdot \bv =0$, we have
\[
 \nabla^\perp\psi(\bx)  = \bv(\bx), \ \forall \ \bx \in \R^2_+.
\]
and $\psi$ is independent on the choice of $\Gamma$, where we denote $\nabla^{\perp} =  \left(-\partial_2 , \partial_1   \right)$. 

\smallskip
Our next three lemmas give various estimates of the stream function $\psi$ which are needed in the next subsections. 

\begin{lemma} There exists a constant $N>0$ such that 
\begin{equation} \label{psi-wei-est}
		\left( \int_{-a}^{a} \int_0^b \frac{|\psi(\bx)|^2}{x_2^3} dx_2dx_1 \right)^{1/2} \leq N b^{1/2} \|\bv\|_{L^2([-a, a] \times [0, b], x_2^{-2})},
	\end{equation}
	for any $a>0, b>0$, and $\bv = (v_1, v_2) \in\hdotzig{\upper}$, where $\psi$ is the stream function associated with $\bv$ defined in \eqref{stream-def}.
\end{lemma}

\begin{proof}  Again, without loss of generality, we can assume $\bv = (v_1, v_2) \in C_{c,\sigma}^\infty(\R^2_+)$. As $\psi$ is independent on the choice of $\Gamma$, it follows from \eqref{stream-def}  that
	\[
	\psi(\bx)=-\int_0^{x_2} v_1(0,s)\,ds+\int_0^{x_1} v_2(s,x_2)\,ds.
	\]
Then $\psi(x_1, 0) =0$, and by the fundamental theorem of calculus, we obtain
\[
|\psi(x_1, x_2)| \leq \int_0^{x_2} |\partial_2\psi(x_1, s)| ds \leq \int_0^{x_2} |\bv(x_1, s)| ds.
\]
From this and the H\"{o}lder inequality,  it follows  that
\[
|\psi(x_1, x_2)|^2  \leq N x_{2}^{3}  \int_0^{x_2}|\bv(x_1, s)|^2 s^{-2} ds \leq N x_{2}^{3}  \int_0^{b}|\bv(x_1, s)|^2 s^{-2} ds
\]
for $(x_1, x_2) \in [-a, a] \times [0, b]$. Hence, by dividing both sides of the estimates by $x_2^{3}$ and then integrating the result on $[-a, a] \times [0, b]$, we obtain
\[
\left( \int_{-a}^{a} \int_0^b \frac{|\psi(\bx)|^2}{x_2^3} dx_2dx_1 \right)^{1/2} \leq N b^{1/2} \left(\int_{-a}^a \int_0^{b}|\bv(x_1, s)|^2 s^{-2} ds dx_1\right)^{1/2}.
\]
The proof of  \eqref{psi-wei-est} is completed.
\end{proof}

In Lemmas \ref{lemma-psi} and \ref{lemma-stream-est-chi} below, we provide necessary estimates for the stream functions associated with vector fields in $\hdotzig{\R^2_+}\cap\linf{\R^2_+,\omega^{1/2}}$ and in $\hdotzig{\R^2_+}\cap\mathcal X$, respectively.

\begin{lemma} \label{lemma-psi} There exists a constant $N>0$ such that 
if  $\bv = (v_1, v_2) \in\hdotzig{\upper}\cap\linf{\R^2_+,\omega^{1/2}}$ with $\omega \in \{\omega_0, \omega_1, \omega_2\}$, we have
	\begin{equation} \label{psi-pointwise}
		|\psi(\bx)| \leq N \big(\|\nabla\bv\|_{L^2(\mathbb{R}^2_+)} + \|\bv\|_{L^\infty(\mathbb{R}^2_+, \omega^{1/2})}\big) g_0(\bx), \quad \forall \bx \in \mathbb{R}^2_+,
	\end{equation}
	where  $\psi$ is the stream function associated with $\bv$ defined in \eqref{stream-def}, and
	\[
	g_0(\bx) = (1+ |x_1|^{1/2})|x_2|^{1/2}\chi_{\{x_2 \leq1\}}(x_2)  + \Big(\ln(x_2)+ \frac{|x_1|}{x_2}\Big)\chi_{\{x_2>1\}}(x_2) ,\]
	for $\bx = (x_1, x_2) \in \mathbb{R}^2_+$. 
\end{lemma} 
\begin{proof} Due to the density of $C_{c, \sigma}^\infty(\R^2_+)$ in $\hdotzig{\upper}$, it is sufficient to prove the lemma for $\bv = (v_1, v_2) \in C_{c, \sigma}^\infty(\R^2_+)$. As $\psi$ is independent on the choice of $\Gamma$, it follows from \eqref{stream-def}  that
	\[
	\psi(\bx)=-\int_0^{x_2} v_1(0,s)\,ds+\int_0^{x_1} v_2(s,x_2)\,ds.
	\]
 We first note that by the the trace inequality and Poincar\'{e} inequality in the square $(0,1) \times (0,1)$, we see that
	\[
	\begin{split}
		\|v_1(0, \cdot)\|_{L^2((0,1))} & \leq N \big( \|v_1\|_{L^2((0,1) \times (0,1))} + \|\nabla v_1\|_{L^2((0,1) \times (0,1))}\big) \\
		& \leq N \|\nabla v_1\|_{L^2((0,1) \times (0,1))} \\
	\end{split}
	\]
	for some absolute constant $N>0$. Hence,
	\begin{equation} \label{est-0226-tra}
		\int_0^{1} |v_1(0, s)| ds \leq \|v_1 (0, \cdot)\|_{L^2((0,1))} \leq  N \|\nabla v_1\|_{L^2(\R^2_+)}.
	\end{equation}
	Now, we consider the case $x_2 \in( 1,\infty)$. Using the definitions of weights in \eqref{weight-def}, \eqref{est-0226-tra}, and with some elementary calculations, we see that
	\begin{align} \notag
		 |\psi(\bx)| &\leq \int_0^{1}|v_1(0,s)|\,ds + \int_1^{x_2}|\bv(0,s)|\,ds+\int_0^{|x_1|}|\bv(s,x_2)|\,ds\\ \notag
		&\leq N \|\nabla v_1\|_{L^2(\R^2_+)} + \int_1^{x_2}|\bv(0,s)|\omega^{1/2}(0,s)\omega^{-1/2}(0,s)\,ds \\ \notag
		& \qquad +\int_0^{|x_1|}|\bv(s,x_2)|\omega^{1/2}(s,x_2)\omega^{-1/2}(s,x_2)\,ds\\ \notag
		&\leq N \|\nabla v_1\|_{L^2(\R^2_+)}  \\ \notag 
		& \quad + N \|\bv\|_\linf{\R^2_+,\omega^{1/2}}\left(\
		\int_1^{x_2}\frac{1}{s}\,ds
		+\int_0^{|x_1|}\frac{1}{x_2^{1/2}(s^2+x_2^2)^{1/4}}\,ds+\int_0^{|x_1|}\frac{1}{x_2}\,ds\right)\\ \notag
		&\leq N \|\nabla v_1\|_{L^2(\R^2_+)} + N\|\bv\|_\linf{\R^2_+,\omega^{1/2}}\left(\ln (x_2) +\int_0^{|x_1|}\frac{1}{x_2}\,ds \right)\\ \label{case1}
		&\leq N \|\nabla v_1\|_{L^2(\R^2_+)} + N\|\bv\|_\linf{\R^2_+,\omega^{1/2}}\left(\ln(x_2) +\frac{|x_1|}{x_2}\right).	\end{align}
	Next, we consider the case $x_2 \in (0,1]$. By H\"{o}lder inequality and \eqref{est-0226-tra} we see that
	\[
	\begin{split}
		\left| \int_0^{x_2} v_1(0, s) ds \right| & \leq \int_0^{x_2} |v_1(0, s)| ds  \leq x_2^{1/2}\|v_1 (0, \cdot)\|_{L^2((0,1))} \\
		& \leq N x_2^{1/2}\|\nabla\bv\|_{L^2(\mathbb{R}^2_+)}.
	\end{split}
	\]
On the other hand, by using  the fundamental theorem of calculus and $v_2(s,0) =0$, we have  
	\[
	v_2(s, x_2) = \int_0^{x_2} \partial_2 v_2(s, t) dt.
	\]
	Therefore, it follows from H\"{o}lder's inequality that
	\[
	|v_2(s, x_2) | \leq \int_0^{x_2}|\partial_2 v_2(s,t)| dt \leq x_2^{1/2} \left(\int_0^1|\partial_2v_2(s,t)|^2 dt\right)^{1/2}.
	\]
	Then, using H\"{o}lder's inequality again, we obtain
	\[
	\begin{split}
		\left|\int_0^{x_1} v_2(s, x_2) ds\right| & \leq x_2^{1/2}\int_0^{|x_1|}\left(\int_0^1|\partial_2v_2(s,t)|^2 dt\right)^{1/2} ds \\
		& \leq x_2^{1/2} |x_1|^{1/2} \left(\int_0^{|x_1|}\int_0^1 | \partial_2v_2(s,t)|^2dsdt \right)^{1/2} \\
		& \leq x_2^{1/2}|x_1|^{1/2}\|\nabla\bv\|_{L^2(\mathbb{R}^2_+)}.
	\end{split}
	\]
	As a result, we have
	\begin{equation} \label{case2}
		|\psi(\bx)| \leq N \|\nabla\bv\|_{L^2(\mathbb{R}^2_+)}(1 + |x_1|^{1/2})x_2^{1/2}, \quad \forall \ \bx = (x_1, x_2) \in \mathbb{R} \times (0,1).
	\end{equation}
	Then, it follows from \eqref{case1} and \eqref{case2} that
	\[
	|\psi(\bx)| \leq N \big(\|\nabla\bv\|_{L^2(\mathbb{R}^2_+)} + \|\bv\|_{L^\infty(\mathbb{R}^2_+, \omega^{1/2})}\big) g_0(\bx),
	\]
	for all $\bx \in \mathbb{R}^2_+$, and the proof of \eqref{psi-pointwise} is completed.
\end{proof}

Next,  recall that  space $\mathcal{X}$ is defined in \eqref{X-def}. We denote
\begin{equation} \label{newspace}
\|v\|_{\mathcal{X}} = \left( \int_{0}^\infty x_2\|v(\cdot, x_2)\|_{L^\infty(\R)}^2 dx_2\right)^{1/2}.
\end{equation}
Then, we also have the following lemma on the estimate of the stream function $\psi$.  \begin{lemma} \label{lemma-stream-est-chi} There is  a constant  $N>0$ such that
\begin{equation} \label{stream-est-chi}
 |\psi(\bx)|  \leq  N \|\nabla v_1\|_{L^2(\R^2_+)}  +  \|v_1\|_{\mathcal{X}} (\ln x_2)^{1/2} \chi_{\{x_2 >1\}}(x_2),
\end{equation}
for every $\bx = (x_1, x_2) \in \R^2_+$ and for $\bv = (v_1, v_2)\in \hdotzig{\upper} \cap \mathcal{X}$, where $\psi$ is the stream function associated with $\bv$ defined in \eqref{stream-def}.
\end{lemma}
\begin{proof} Again, as the definition of $\psi$ in \eqref{stream-def} is independent on the choice of $\Gamma$, we have
\[
\psi(\bx) = -\int_0^{x_2} v_1(x_1,s)\,ds.
\]
Consider $x_2\in (0,1]$. Applying H\"{o}lder inequality and \eqref{L2-01-v} of Lemma \ref{Hardy-2}, we have 
\[ 
\begin{split}
|\psi(\bx)|\leq \int_{0}^{x_2} \big| v_1(x_1,s)\big|\,ds & \leq  x_2^{1/2}\left( \int_{0}^{1} |v_1(x_1, s)|^2\,ds\right)^{1/2} \\
& \leq N x_2^{1/2} \|\nabla v_1\|_{L^2(\R^2_+)}.
\end{split}
\]
This implies \eqref{stream-est-chi}. 

\smallskip
Next, we consider $x_2\in (1,\infty)$.  By H\"{o}lder inequality and \eqref{L2-01-v} of Lemma \ref{Hardy-2} again, we have 
	\begin{align*}
		|\psi(\bx)|&\leq \Big|\int_{0}^{1} v_1(x_1,s)\,ds\Big|+\Big|\int_1^{x_2} v_1(x_1,s)\,ds\Big|\\
		&\leq \left( \int_{0}^{1} |v_1(x_1, s)|^2\,ds\right)^{1/2}+\Big(\int_1^{x_2}\frac{1}{s}\,ds\Big)^{1/2}\Big(\int_1^{x_2} s|v_1(x_1,s)|^2\,ds\Big)^{1/2}\\
		& \leq N\|\nabla v_1\|_{L^2(\R^2_+)}+ (\ln x_2)^{1/2} \Big(\int_1^{\infty} s\|v_1(\cdot, x_2)\|_{L^\infty(\R)}^2\,ds\Big)^{1/2}\\
		&\leq N \|\nabla v_1\|_{L^2(\R^2_+)} +  \|v_1\|_{\mathcal{X}}  (\ln x_2)^{1/2}.
	\end{align*}
Then, \eqref{stream-est-chi} also follows. The proof of the lemma is completed.
\end{proof}
\subsection{Density lemmas} This subsection establishes two density lemmas which are needed to prove Theorem \ref{theorem2.2}. The following lemma proves a density property of $\comsig{\upper}$ in $\hdotzig{\upper}\cap\linf{\R^2_+,\omega^{1/2}}$, which is a consequence of Lemma \ref{lemma-psi}.
\begin{lemma} \label{approximationlemma} Let $\omega \in \{\omega_0, \omega_1, \omega_2\}$. Then, for every $\bv\in\hdotzig{\upper}\cap\linf{\R^2_+,\omega^{1/2}}$, there exists a sequence $\{\bv_n\}_{n\in\N}\subset\comsig{\upper}$ satisfying the following two properties:
\begin{itemize}
\item[\textup{(i)}] The sequence $\{\bv_n\}_{n\in\N}$ converges to $\bv$ in $\hdotzig{\upper}$.
\item[\textup{(ii)}] For any $\bu\in\hdotzig{\upper}$, the sequence $\{\bu\otimes\bv_n\}_{n\in\N}$ converges to $\bu\otimes\bv$ in $\ltwo{\upper}$.
\end{itemize}
\end{lemma}
\begin{proof} Let $\phi\in C^\infty_c(\R)$ to be a standard non-negative smooth cutoff function satisfying 
\[
\phi  = 1\quad \text{on} \quad (-1/2,1/2), \quad \phi =0 \quad \text{on} \quad \R \setminus (-1, 1), \quad \text{and} \quad  \|\phi\|_{C^2(\R)}\leq C_0.
\]
Also, let $\varphi \in C^\infty(\R)$ satisfying  
\[
\varphi  =0 \quad  \text{on} \quad [-1,1], \quad \varphi =1 \quad \text{on} \quad \R\setminus(-2,2), \quad  \text{and} \quad  \|\varphi\|_{C^2(\R)}\leq C_0.
\]
Note that here, $C_0>0$ is some generic constant. Now, let $M_0>e$ be fixed and for $n>e^e-M_0$ and $\bx=(x_1,x_2)\in\R^2_+$, we define 
\[
\eta_n(\bx)=\phi_n(\bx) \varphi(n x_2), 
\]
where
\[
\phi_n(\bx) = \phi\left(\frac{\ln(\ln(|\bx|+M_0))}{\ln(\ln(n+M_0))}\right).
\]
Observe that 
\begin{align*}
\phi_n(\bx)=\left\{
\begin{array}{cccl}
   &1& \quad &\text{when} \quad |\bx|\leq\gamma_n =e^{(\ln(n+M_0))^{1/2}}-M_0\\
   &0& \quad &\text{when} \quad |\bx|\geq n
   \end{array} \right.
\end{align*}
and 
\begin{align*}
\varphi(n x_2)=\left\{
\begin{array}{cccl}
   &  0& \quad &\text{when} \quad |x_2|\leq1/n\\
   &  1 & \quad &\text{when} \quad |x_2|\geq 2/n.
   \end{array} \right.
\end{align*}
Moreover, it follows from the product rule and the definition of $\eta_n$ that there is a generic constant $N>0$ such that
\begin{equation} \label{eqn-0302-1}
\begin{split}
|\nabla \eta_n (\bx) | & \leq \frac{N \chi_{A_n}(\bx)}{\ln(\ln(n+M_0))\ln(|\bx|+M_0)(|\bx|+M_0)}  +n N\chi_{I_n}(\bx),
\end{split}
	\end{equation}
and
\begin{equation} \label{eqn-0302-1-b}
	\begin{split}
|\nabla^2\eta_n (\bx)| &\leq \dfrac{N \chi_{A_n}(\bx)}{\ln(\ln(n+M_0))\ln(|\bx|+M_0)(|\bx|+M_0)|\bx|}\\
	& \qquad +\dfrac{n N \displaystyle{\chi_{A_n \cap I_n}}(\bx)}{\ln(\ln(n+M_0))\ln(|\bx|+M_0)(|\bx|+M_0)}\\
	& \qquad +n^2 N\chi_{I_n}(\bx),
	\end{split}
	\end{equation}
 for $\bx=(x_1, x_2) \in \R^2_+$ and for all $n > e^e -M_0$, where
 \begin{equation} \label{An-In}
 \begin{split}
 A_n  & = \big\{\bx = (x_1, x_2) \in \R^2_+: \gamma_n < |\bx|< n,  1/n < x_2 \big \}, \\
 I_n & =\big\{\bx = (x_1, x_2) \in \R^2_+: |\bx| \leq n,  x_2 \in (1/n, 2/n) \big\}.
 \end{split}
 \end{equation}
 	
Next, for the given $\bv \in\hdotzig{\upper}\cap\linf{\R^2_+,\omega^{1/2}}$, let $\psi$ be the stream function associated with $\bv$ defined in \eqref{stream-def}. We define
	\[
	\bv_n(\bx)= \nabla^\perp(\eta_n(\bx)\psi(\bx)), \quad \bx = (x_1, x_2) \in \mathbb{R}^2_+, 
	\]
	where $\nabla^{\perp} =  \left(-\partial_2 , \partial_1   \right)$. 
	We note that $\bv_n$ is compactly supported in $\R^2_+$ and it is divergence free. By the standard mollification, it is sufficient to prove the assertions of the lemma for the sequence $\{\bv_n\}$.
	Note that the constants $N>0$ used for the estimates in this proof change line by line and they are independent  of $n$. 

\smallskip	
We begin with proving (ii). Let $\bu\in\hdotzig{\upper}$ be fixed and we prove that
	\begin{equation} \label{lim-asser-1}
		\lim_{n\rightarrow \infty}  \|\bu\otimes(\bv-\bv_n)\|_\ltwo{\R^2_+} =0.
	\end{equation}
	Note that $\bv - \bv_n =(1-\eta_n)\bv - \psi \nabla^\perp(\eta_n)$ and 
	\begin{equation} \label{eqn-0303-2}
		\|\bu\otimes(\bv-\bv_n)\|_\ltwo{\R^2_+}\leq\|(1-\eta_n)\bu\otimes\bv\|_\ltwo{\R^2_+}+\|\psi \bu\otimes\nabla\eta_n\|_\ltwo{\R^2_+}.
	\end{equation}
	We then treat the two terms in the right hand side of \eqref{eqn-0303-2}. For the first term, we observe that
	\begin{equation}   \label{uv-finite-linfty}
	\begin{split}
		& \|\bu\otimes\bv\|_\ltwo{\R^2_+}  \leq \|\bv\|_{L^\infty(\R^2_+, \omega^{1/2})}\|\bu \|_\ltwo{\R^2_+, \omega^{-1}} \\
		&  \leq N \|\bv\|_{L^\infty(\R^2_+, \omega^{1/2})}\|\nabla \bu \|_\ltwo{\R^2_+} <\infty.    
		\end{split}
	\end{equation}
	Therefore, by the Lebesgue dominated convergence theorem, we obtain
	\begin{equation} \label{lim-1uv}
		\lim_{n\rightarrow \infty} \|(1-\eta_n)\bu\otimes\bv\|_\ltwo{\R^2_+} =0.
	\end{equation}
On the other hand, by \eqref{eqn-0302-1} and the triangle inequality, it follows that
 \[
 \begin{split}
 &   \|\psi \bu\otimes\nabla\eta_n\|_\ltwo{\R^2_+} \\
  &\leq \frac{N}{\ln(\ln(n+M_0))} \left(\int_{A_n}\frac{ |\psi(\bx)|^2 |\bu (\bx)|^2}{[\ln(|\bx|+M_0)]^2(|\bx|+M_0)^2} \,d\bx\right )^{1/2}\\
    &\qquad  +  n N \left( \int_{I_n} |\psi(\bx)|^2 |\bu (\bx)|^2 \,d\bx \right)^{1/2}\\
    &=K_1+K_2.
 \end{split}
\]
Then, by using \eqref{psi-pointwise} of Lemma \ref{lemma-psi}, we estimate $K_1$ by the following calculations
 \begin{align*}
     K_1^2 &= \frac{N}{[\ln(\ln(n+M_0))]^2}\left(\int_{A_n} \dfrac{(1+|x_1|^{1/2})^2x_{2}\chi_{\{1/n<x_2\leq 1\}}(x_2)|\bu(\bx)|^2}{ [\ln(|\bx|+M_0)]^2(|\bx|+M_0)^2} d\bx\right. \\
     &\quad\left.+ \int_{A_n} \dfrac{(\ln(x_2)+|x_1|x_2^{-1})^2\chi_{\{x_2> 1\}}(x_2)|\bu(\bx)|^2}{ [\ln(|\bx|+M_0)]^2(|\bx|+M_0)^2} d\bx \right) \\
     & \leq \frac{N}{[\ln(\ln(n+M_0))]^2}\left(\int_{A_n} \dfrac{(1+|x_1|)\chi_{\{x_2\leq 1\}}(x_2)|\bu(\bx)|^2}{ [\ln(|\bx|+M_0)]^2(|\bx|+M_0)^2x_2^2} d\bx \right. \\
     &\quad\left.+ \int_{A_n} \dfrac{|\bu(\bx)|^2}{ (|\bx|+M_0)^2} +\dfrac{|\bu(\bx)|^2}{ [\ln(|\bx|+M_0)]^2x_{2}^{2}} d\bx \right)\\
    &\leq \frac{N}{[\ln(\ln(n+M_0))]^2}\left(\frac{1}{[\ln(M_0)]^2M_0} + 1+\frac{1}{[(\ln(M_0)]^2}\right) \|\bu\|_\ltwo{\balln,x_{2}^{-2}}.
 \end{align*}
By the Hardy inequality (Lemma \ref{Hardy}), we have that
\begin{equation}\label{DCT}
	\|\bu\|_{L^2(\R^2_+, x_2^{-2})} \leq N \|\nabla \bu\|_{L^2(\R^2)} < \infty.
\end{equation}
Consequently, 
\[ 
\lim_{n\rightarrow \infty} K_1 =0.
\]
Similarly, we estimate $K_2$
 \begin{align*}
     K_2^2 &=n^2 N \int_{I_n} |\psi (\bx) |^2 |\bu(\bx)|^2\,d\bx\\
     &\leq n^2 N\int_{I_n} (1+|x_1|^{1/2})^2x_{2} \frac{|\bu (\bx)|^2x_{2}^2}{x_{2}^2}\,d\bx\\
     &\leq n^2 N\int_{I_n} (1+n^{1/2})^2 \frac{1}{n^3} \frac{|\bu(\bx)|^2}{x_{2}^2}\,d\bx\\
     &\leq \frac{N(1+n)}{n} \int_{I_n}  \frac{|\bu(\bx)|^2}{x_{2}^2}\,d\bx\\
     &\leq \frac{N(1+n)}{n} \|\bu\|_{\ltwo{I_n ,x_{2}^{-2}}}\leq N \|\bu\|_{\ltwo{\R^2_+\cap\{0<x_2<2/n\},x_{2}^{-2}}}.
 \end{align*}
Therefore, using \eqref{DCT} and the Lebesgue dominated convergence theorem, we conclude that
 $K_2$ converges to $0$ as $n$ goes to infinity. Thus,
	\begin{equation}\label{lim-2uv}
		\lim_{n\rightarrow \infty} \|\psi \bu\otimes\nabla\eta_n\|_\ltwo{\R^2_+} =0.
	\end{equation} 
	Then, the assertion \eqref{lim-asser-1} follows from \eqref{lim-1uv} and \eqref{lim-2uv}, and the proof of (ii) is completed.
	
\smallskip 
It now remains to prove (i), namely 
	\begin{equation} \label{lim-asser-2}
		\lim_{n\rightarrow \infty} \| \nabla \bv_n - \nabla \bv\|_{L^2(\R^2_+)} =0.
	\end{equation}
We note that
\begin{align} \notag
& \| \nabla \bv_n - \nabla \bv \|_{L^2(\R^2_+)} \\ \label{con-nabla}
	& \leq \|(1-\eta_n) \nabla \bv\|_{L^2(\R^2_+)} + 2 \|\nabla \eta_n \otimes\bv\|_{L^2(\R^2_+)} + \|\psi \nabla^2\eta_n\|_{L^2(\R^2_+)}.
\end{align}
As before, by using the Lebesgue dominated convergence theorem, we see that
	\begin{equation} \label{lim-2-1}
		\lim_{n\rightarrow \infty}  \|(1-\eta_n) \nabla \bv\|_{L^2(\R^2_+)} =0.
	\end{equation}
Next, using \eqref{eqn-0302-1}, we control the second term in the right hand side of \eqref{con-nabla} by 
\begin{align*}
    \|\nabla \eta_n \otimes\bv\|_{L^2(\R^2_+)}&\leq \frac{N}{\ln(\ln(n+M_0))}\left(\int_{A_n}  \dfrac{|\bv(\bx)|^2}{[\ln(|\bx|+M_0)]^2(|\bx|+M_0)^2}\,d\bx\right)^{1/2}\\
    &\qquad + n N \left( \int_{I_n}  |\bv(\bx) |^2 \,d\bx \right)^{1/2}\\
    &=J_1+J_2.
\end{align*}
Observe that 
\begin{align*}
    J_1^2& = \dfrac{N}{[\ln(\ln(n+M_0))]^2}\int_{A_n} \dfrac{|\bv(\bx)|^2}{[\ln(|\bx|+M_0)]^2(|\bx|+M_0)^2}\,d\bx\\
    &\leq \dfrac{N}{[\ln(\ln(n+M_0))]^2[\ln(M_0)]^2}\int_{A_n} \dfrac{|\bv(\bx) |^2}{x_2^2}\,d\bx\\
    &\leq \dfrac{N}{[\ln(\ln(n+M_0))]^2[\ln(M_0)]^2} \|\bv\|^2_\ltwo{\R^2_+,x_2^{-2}}.
\end{align*}
Therefore, $J_1$ goes to $0$ as $n$ goes to infinity. In addition, by a simple manipulation, we estimate $J_2$ as
\begin{align*}
    J_2^2 & = N n^2 \int_{I_n} \dfrac{x_2^2|\bv(\bx)|^2}{x_2^2} \,d\bx \leq N \|\bv\|^2_\ltwo{\R \times (0, 2/n)\}, x_2^{-2}}.
\end{align*}
Once again, we use \eqref{DCT} and the Lebesgue dominating convergence theorem to conclude that $J_2$ approaches $0$ as $n$ goes to infinity.  Hence, we have proved that
	\begin{equation} \label{lim-2-2}
		\lim_{n\rightarrow \infty} \|\nabla \eta_n \otimes\bv\|_{L^2(\R^2_+)} =0.
	\end{equation}

\smallskip
 Next, we consider the term $\|\psi\nabla ^2\eta_n\|_{L^2(\R^2_+)}$ on the right hand side of \eqref{con-nabla}.  By \eqref{psi-pointwise} and  \eqref{eqn-0302-1-b}, it follows that
	\begin{align*} 
	&  \|\psi \nabla ^2\eta_n\|_{L^2(\R^2_+)}^2 \\  
		& \leq \frac{N}{[\ln(\ln(n+M_0)]^2} \int_{A_n}\dfrac{x_2(1+|x_1|^{1/2})^2\chi_{\{x_2\leq1\}(x_2)}}{(\ln(|\bx|+M_0))^2(|\bx|+M_0)^2|\bx|^2} d\bx \\ 
		& \quad + \frac{N}{[\ln(\ln(n+M_0)]^2}\int_{A_n}\dfrac{(\ln(x_2)+|x_1|x_2^{-1})^2\chi_{\{x_2>1\}}}{[\ln(|\bx|+M_0)]^2(|\bx|+M_0)^2|\bx|^2} d\bx  \\
   &\quad+ N n^2\int_{A_n\cap I_n} |\psi(\bx)|^2d\bx + Nn^4 \int_{I_n}|\psi (\bx)|^2\, d\bx \\
		& = G_1  + G_2 + G_3 + G_4. 
	\end{align*}
We control $G_1$ by the following calculations
\begin{align*}
G_1 &=  \frac{N}{[\ln(\ln(n+M_0)]^2}  \int_{A_n}\dfrac{x_2(1+|x_1|^{1/2})^2\chi_{\{x_2\leq1\}(x_2)}}{(\ln(|\bx|+M_0))^2(|\bx|+M_0)^2|\bx|^2} d\bx\\
   & \leq \dfrac{N}{[\ln(\ln(n+M_0)]^2\gamma_n^2}\int_{A_n}\dfrac{1}{(\ln(|\bx|+M_0))^2(|\bx|+M_0)^2}d\bx\\
   & \quad+ \dfrac{N}{[\ln(\ln(n+M_0)]^2\gamma_n}\int_{A_n}\dfrac{1}{(\ln(|\bx|+M_0))^2(|\bx|+M_0)^2}d\bx\\
   & \leq \dfrac{N}{[\ln(\ln(n+M_0)]^2\gamma_n}\int_{\R^2_+}\dfrac{1}{(\ln(|\bx|+M_0))^2(|\bx|+M_0)^2}d\bx.
\end{align*}
Note that the last integral is finite which can be verified by  using polar coordinates.  Therefore,
\[
\lim_{n\rightarrow \infty} G_1 =0.
\]
\smallskip
In the same manner, the estimate for $G_2$ is shown below
\begin{align*}
G_2  &\leq \frac{N}{[\ln(\ln(n+M_0))]^2}\int_{\R^2_+}\dfrac{1}{(|\bx|^2+1)^2} d\bx\\
& \rightarrow 0 \quad \text{as} \quad n \rightarrow \infty.
\end{align*}
\smallskip
Next, using \eqref{psi-wei-est}, we control $G_4$ as 
\begin{align*}
    G_4&\leq N n^4\int_{I_n}  |\psi(\bx)|^2 \,d\bx\leq N n^4\int_{I_n} \dfrac{|\psi(\bx) |^2x_2^3}{x_2^3}\,d\bx\\
    &\leq N n\int_{-\infty}^\infty\int_0^{2/n}\dfrac{|\psi(x_1,x_2)|^2}{x_2^3}\,dx_2\,dx_1\leq N \|\bv\|^2_\ltwo{\R\times (0,2/n),x_2^{-2}}.
\end{align*}
Observe that from Hardy's inequality (Lemma \ref{Hardy})
\[
\|\bv\|_\ltwo{\R\times (0,2/n),x_2^{-2}}\leq\|\bv\|_\ltwo{\R^2_+,x_2^{-2}}\leq N \|\nabla \bv\|_{L^2(\R^2_+)} <\infty.
\]
Therefore, by the Lebesgue dominating convergence theorem we have $G_4$ goes to $0$ as $n$ goes to infinity.  Note that as $G_3 \leq G_4$, we conclude that
\[
\lim_{n\rightarrow \infty} \big(G_3 + G_4\big) =0.
\]
Collecting the limits of $G_k$ for $k =1,2,3, 4$ when $n \rightarrow \infty$ that we just proved, we have
	\begin{equation} \label{lim-2-3}
		\lim_{n\rightarrow \infty} \| \psi \nabla ^2\eta_n\|_{L^2(\R^2_+)}^2 =0.
	\end{equation}
Then, the assertion \eqref{lim-asser-2} follows from \eqref{lim-2-1}, \eqref{lim-2-2}, and \eqref{lim-2-3}.
\end{proof}

The next lemma proves the density type result of $\comsig{\upper}$ in $\hdotzig{\upper}\cap\mathcal X$ for which Lemma \ref{lemma-stream-est-chi} is used.  
\begin{lemma} \label{approximationlemma2} For every $\bv\in\hdotzig{\upper}\cap\mathcal{X}$, there exists a sequence $\{\bv_n\}_{n\in\N}\subset\comsig{\upper}$ satisfying the following properties:
	\begin{itemize}
		\item[\textup{(i)}] The sequence $\{\bv_n\}_{n\in\N}$ converges to $\bv$ in $\hdotzig{\upper}$.
		\item[\textup{(ii)}] For any $\bu\in\hdotzig{\upper}$, the sequence $\{\bu\otimes\bv_n\}_{n\in\N}$ converges to $\bu\otimes\bv$ in $\ltwo{\upper}$.
	\end{itemize}
\end{lemma}
\begin{proof} The proof is similar to that of Lemma \ref{approximationlemma} and we only outline some main steps. Let $\eta_n$ be defined as in the proof of Lemma \ref{approximationlemma}. For the given $\bv \in \hdotzig{\upper}\cap\mathcal{X}$, let $\psi$ be defined in \eqref{stream-def}  and let
		\[
		\bv_n(\bx)= \nabla^\perp(\eta_n(\bx)\psi(\bx)), \quad \bx = (x_1, x_2) \in \mathbb{R}^2_+.
		\]
		It is sufficient to prove the assertions of the lemma for the sequence $\{\bv_n\}_{n\in\N}$.
				
		\smallskip	
We begin with proving (ii). Let $\bu\in\hdotzig{\upper}$ be fixed and we prove that
		\begin{equation} \label{lim-asser-1-chi}
			\lim_{n\rightarrow \infty}  \|\bu\otimes(\bv-\bv_n)\|_\ltwo{\R^2_+} =0.
		\end{equation}
		Note that $\bv - \bv_n =(1-\eta_n)\bv - \psi \nabla^\perp(\eta_n)$ and 
		\begin{equation} \label{eqn-0303-2-chi}
			\|\bu\otimes(\bv-\bv_n)\|_\ltwo{\R^2_+}\leq\|(1-\eta_n)\bu\otimes\bv\|_\ltwo{\R^2_+}+\|\psi \bu\otimes\nabla\eta_n\|_\ltwo{\R^2_+}.
		\end{equation}
		We then treat the two terms in the right hand side of \eqref{eqn-0303-2-chi}. For the first term, we observe that 
		\begin{align} \notag
					\|\bu\otimes\bv\|_\ltwo{\R^2_+} & \leq \left(\int_{0}^{\infty}\int_{\R} |\bu(\bx)|^2|\bv(\bx)|^2\,dx_1\,dx_2\right)^{1/2} \\
			&    \leq \left(\int_{0}^{\infty} \|\bv(\cdot,x_2)\|_{\linf{\R}}^{2}\int_{\R} |\bu(\bx)|^2\,dx_1\,dx_2\right)^{1/2} \notag \\
			&    \leq \left(\int_{0}^{\infty} \|\bv(\cdot,x_2)\|_{\linf{\R}}^{2}\|\bu(\cdot,x_2)\|_{L^2(\R)}^2\,dx_2\right)^{1/2} \notag \\
				&    \leq \left(\int_{0}^{\infty} \|\bv(\cdot,x_2)\|_{\linf{\R}}^{2}(x_2^{1/2}\|\nabla\bu(\bx)\|_{L^2(\R_{+}^{2})})^2\,dx_2\right)^{1/2} \notag\\
				&\leq \|\nabla\bu(\bx)\|_{L^2(\R_{+}^{2})}\left(\int_{0}^{\infty} \|\bv(\cdot,x_2)\|_{\linf{\R}}^{2}x_2\,dx_2\right)^{1/2}  \notag\\
 & = \|\nabla\bu(\bx)\|_{L^2(\R_{+}^{2})} \| \bv\|_{\mathcal{X}} < \infty,  \label{uv-finite-chi}
						\end{align}
		where we have used  \eqref{3.3-2} of Lemma \ref{Hardy-2} in the fourth step of the previous estimate.
		Therefore, by the Lebesgue dominated convergence theorem, we obtain
		\begin{equation} \label{lim-1uv-chi}
			\lim_{n\rightarrow \infty} \|(1-\eta_n)\bu\otimes\bv\|_\ltwo{\R^2_+} =0.
		\end{equation}
On the other hand, by using Lemma \ref{lemma-stream-est-chi} instead of Lemma \ref{lemma-psi}, we can follow exactly as in the proof Lemma \ref{approximationlemma} to prove
\begin{equation}\label{lim-2uv-chi}
	\lim_{n\rightarrow \infty} \|\psi \bu\otimes\nabla\eta_n\|_\ltwo{\R^2_+} =0.
\end{equation} 
Then, the assertion \eqref{lim-asser-1-chi} follows from \eqref{lim-1uv-chi} and \eqref{lim-2uv-chi}, and the proof of (ii) is completed.

\smallskip 
Next, we prove (i), namely 
\begin{equation} \label{lim-asser-2-chi}
	\lim_{n\rightarrow \infty} \| \nabla \bv_n - \nabla \bv\|_{L^2(\R^2_+)} =0.
\end{equation}
Note that
\begin{align} \notag
	& \| \nabla \bv_n - \nabla \bv \|_{L^2(\R^2_+)} \\ \label{con-nabla-chi}
	& \quad \leq \|(1-\eta_n) \nabla \bv\|_{L^2(\R^2_+)} + 2 \|\nabla \eta_n \otimes\bv\|_{L^2(\R^2_+)} + \|\psi \nabla^2\eta_n\|_{L^2(\R^2_+)}.
\end{align}
The first two terms in \eqref{con-nabla-chi} can be proved to go to $0$ as shown in the proof of Lemma \ref{approximationlemma}. We then consider the term $\|\psi\nabla ^2\eta_n\|_{L^2(\R^2_+)}$ on the right hand side of \eqref{con-nabla}.  By \eqref{stream-est-chi} and  \eqref{eqn-0302-1-b}, it follows that that
\begin{align*} 
	&  \|\psi \nabla ^2\eta_n\|_{L^2(\R^2_+)}^2 \\  
	& \leq \frac{N}{[\ln(\ln(n+M_0)]^2} \int_{A_n}\dfrac{(1+|\ln x_2|^{1/2})^2\chi_{\{x_2\geq1\}(x_2)}}{(\ln(|\bx|+M_0))^2(|\bx|+M_0)^2|\bx|^2} d\bx \\ 
	& \quad + \frac{N}{[\ln(\ln(n+M_0)]^2}\int_{A_n}\dfrac{\chi_{\{x_2\leq1\}}}{[\ln(|\bx|+M_0)]^2(|\bx|+M_0)^2|\bx|^2} d\bx  \\
	&\quad+ N n^2\int_{A_n\cap I_n} |\psi(\bx)|^2d\bx + Nn^4 \int_{I_n}|\psi (\bx)|^2\, d\bx \\
	& = G_1  + G_2 + G_3 + G_4,
\end{align*}
where $A_n$ and $I_n$ are defined in \eqref{An-In}. We control $G_1$ by the following calculations
\begin{align*}
	G_1 &=  \frac{N}{[\ln(\ln(n+M_0)]^2}  \int_{A_n}\dfrac{(1+|\ln x_2|^{1/2})^2\chi_{\{x_2\geq1\}(x_2)}}{(\ln(|\bx|+M_0))^2(|\bx|+M_0)^2|\bx|^2} d\bx\\
	& \leq \dfrac{N}{[\ln(\ln(n+M_0)]^2}\int_{A_n}\dfrac{1}{(\ln(|\bx|+M_0))^2(|\bx|+M_0)^2|\bx|^2}d\bx\\
	& \quad+ \dfrac{N}{[\ln(\ln(n+M_0)]^2}\int_{A_n}\dfrac{1}{\ln(|\bx|+M_0)(|\bx|+M_0)^2|\bx|^2}d\bx\\
	& \leq \dfrac{N}{[\ln(\ln(n+M_0)]^2}\int_{\R^2_+}\dfrac{1}{(\ln(|\bx|+M_0))(|\bx|+M_0)^2|\bx|^2}d\bx.
\end{align*}
Similarly,
\begin{align*}
	G_2  &\leq \frac{N}{[\ln(\ln(n+M_0))]^2}\int_{\R^2_+}\dfrac{1}{(\ln(|\bx|+M_0))^2(|\bx|+M_0)^2|\bx|^2}d\bx.
\end{align*}
Therefore,
\[
\lim_{n\rightarrow \infty} [G_1 + G_2] =0.
\]
With the same calculation, we also obtain
\[
\lim_{n\rightarrow \infty} [G_3 + G_4] =0.
\]
Then, collecting the estimates for $G_k$ for $k =1,2,3, 4$, we have
\begin{equation*}  
\lim_{n\rightarrow \infty} \| \psi \nabla ^2\eta_n\|_{L^2(\R^2_+)}^2 =0,
\end{equation*}
and the assertion \eqref{lim-asser-2-chi} follows. The proof of the lemma is completed.
\end{proof}

\section{Existence and uniqueness of weak solutions} \label{existence-sec}
This section is divided into two subsections and each of which provides the proof of each of  the Theorem \ref{theorem2.1} and Theorem \ref{theorem2.2}. For convenience in writing, we use $\langle \cdot,\cdot \rangle$ to denote the inner product in Euclidean space $\R^2$. Similarly we write
\[
\langle\nabla\bu,\nabla\bphi\rangle_\ltwo{\R^2_+} = \sum_{i=1}^2\int_{\R^2_+}\nabla u_i\cdot\nabla\varphi_i\,d\bx.
\]
From this, \eqref{1.2} is rewritten as 
\begin{equation} \label{1.3}
\langle\nabla\bu,\nabla\bphi\rangle_\ltwo{\R^2_+}+\langle\bu\cdot\nabla\bu,\bphi\rangle_\ltwo{\R^2_+}= \langle\bF,\nabla\bphi\rangle_\ltwo{\R^2_+}, 
\end{equation}
for all $\bphi \in C^\infty_{c,\sigma}(\R^2_+)$.
\subsection{Existence of solutions}  This  proves Theorem \ref{theorem2.1}.  Observe that Theorem \ref{theorem2.1} could be well-known, but we prove it here for completeness. We follow the method of invasion domain introduced by J. Leray in \cite{Leray}. For each $n \in \mathbb{N}$, we consider \eqref{1.1} in the upper-half ball $B_n^+$  
\begin{equation} \label{1.1-ball}
\left\{
\begin{array}{cccl}
   - \boldsymbol\Delta\bu + (\bu\cdot\nabla) \bu + \nabla P +  \text{div}(\bF) & = &  0 & \quad \text{in} \quad B_n^+,\\
    \nabla\cdot\bu & = & 0 & \quad \text{in} \quad  B_n^+,\\
   \qquad \bu  & = & 0 & \quad \text{on} \quad  \p B_n^+.
   \end{array} \right.
 \end{equation}
We begin with the following lemma on the existence of weak solutions of \eqref{1.1-ball}.
\begin{lemma} \label{theorem4.1} For every $n \in \mathbb{N}$ and  $\bF\in\ltwo{B_n^+}^{2\times2}$, there exists a weak solution $\bu_n\in\hdotzig{\balln}$ to the Navier-Stokes equations \eqref{1.1-ball}  satisfying
\begin{equation} \label{energy-equality}
    \|\nabla \bu_n\|_\ltwo{\balln}^2  = \langle \bF, \nabla\bu_n \rangle_\ltwo{\balln}.
\end{equation}
\end{lemma}
\begin{proof}  The proof is standard, see \cite{Lad} for instance. We provide the details for completeness. We will show that there is $\bu_n\in\hdotzig{\balln}$ satisfying
\begin{equation} \label{4.1}
\langle\nabla\bu_n,\nabla\bphi\rangle_\ltwo{\balln}+\langle\bu_n\cdot\nabla\bu_n,\bphi\rangle_\ltwo{\balln}=\langle\bF,\nabla\bphi\rangle_\ltwo{\balln},
\end{equation}
for all $\bphi \in \hdotzig{\balln}$. We use the Leray-Schauder fixed point theorem (see \cite[Theorem 10.6 p. 228]{gilbarg}, for example).  To this end, we look for $\bu_n \in \hdotzig{\balln}$ as a fixed point of an equation, which is equivalent to \eqref{4.1}.

We first note that as $B_n^+$ is bounded, due to the Poincare's inequality, the space $\hdotzig{\balln}$ is the Hilbert space with the inner product
\[
\langle \bu,  \bv  \rangle_\hdotzig{\balln} = \int_{B_n^+} \langle \nabla \bu(\bx),   \nabla \bv(\bx)\rangle d\bx, \quad \forall \bu, \bv \in \hdotzig{\balln}.
\]
Moreover, note that
\begin{align*}
    \left|\langle\bF,\nabla\bphi\rangle_\ltwo{\balln}\right| & \leq \|\bF\|_\ltwo{\balln}\|\nabla\bphi\|_\ltwo{\balln} \\
&  = \|\bF\|_\ltwo{\balln} \|\bphi\|_\hdotzig{\balln},
\end{align*}
for any $\bphi \in \hdotzig{\balln}$. Then, by the Riesz representation theorem, there exists a unique $\bg \in\hdotzig{\balln}$ such that 
\begin{align*}
    \langle\bg ,\bphi\rangle_\hdotzig{\balln}=\langle\bF,\nabla\bphi\rangle_\ltwo{\balln}, \quad \bphi\in\hdotzig{\balln}.
\end{align*}

\smallskip
Next, note that for given $\bPhi,\bpsi,  \in \lfour{\balln}^2$, we have
\[
\begin{split}
\left|\langle\bPhi\cdot\nabla\bphi,\bpsi\rangle_\ltwo{\balln}\right| &\leq\|\bPhi\otimes\bpsi\|_\ltwo{\balln}\|\nabla\bphi\|_\ltwo{\balln}\\
         &\leq\|\bPhi\|_\lfour{\balln}\|\bpsi\|_\lfour{\balln}\|\bphi\|_\hdotzig{\balln}, 
 \end{split}
\]
for any $\bphi \in \hdotzig{\balln}$.
Then, by the Riesz representation theorem again, there is a bilinear bounded map
\[
G: \lfour{\balln}^2 \times \lfour{\balln}^2 \rightarrow \hdotzig{\balln}
\]
 defined by
\begin{align*}
    \langle  G(\bPhi,\bpsi),\bphi\rangle_\hdotzig{\balln} & = \langle\bPhi\cdot\nabla\bphi,\bpsi\rangle_\ltwo{\balln}, \quad \bphi \in \hdotzig{\balln}.
\end{align*}
Now, let $T:\hdotzig{\balln}\rightarrow\hdotzig{\balln}$ defined by 
\[
T(\bPhi):=-G(\bPhi,\bPhi), \quad \bPhi \in \hdotzig{\balln}.
\]
Observe that by the integration by parts, we also have
\begin{equation} \label{T-def}
 \langle  T(\bPhi),\bphi\rangle_\hdotzig{\balln} =\langle\bPhi\cdot\nabla\bPhi,\bphi\rangle_\ltwo{\balln}, \quad \bphi \in \hdotzig{\balln}.
\end{equation}
On the other hand, as $\hdotzig{\balln}$ is compactly embedded in $\lfour{\balln}^2$. We see that the map $T:\hdotzig{\balln}\rightarrow\hdotzig{\balln}$ is compact.

Now, note that \eqref{4.1} is equivalent to
\begin{equation} \label{n-eqn-0304}
    \langle\bu_n+T(\bu_n)-\bg,\bphi\rangle_\hdotzig{\balln}=0, \quad \bphi \in \hdotzig{\balln}
\end{equation}
and we use the Leray-Schauder fixed point theorem to show that there is a solution $\bu_n \in \hdotzig{\balln}$ to the equation  \eqref{n-eqn-0304}. 

For each $\lambda \in [0,1]$, let $ T_\lambda  :  \hdotzig{\balln}\times[0,1]\rightarrow \hdotzig{\balln}$ be defined by
\[
T_{\lambda}(\bv) = -\lambda(T(\bv)-\bg), \quad \bv \in \hdotzig{\balln}.
\]
We claim the following:
\begin{itemize}
\item[\text{(i)}] $T_\lambda$ is a compact map.
\item[\text{(ii)}]  There is $M>0$ such that $\|\bv\|_\hdotzig{\balln}\leq M$ whenever $T_{\lambda}(\bv)=\bv$ for some $\lambda\in[0,1]$.
\item[\text{(iii)}]  $T_0(\bv)=0$  for all  $\bv\in\hdotzig{\balln}$.
\end{itemize}

\smallskip
Indeed, we see that (i) follows directly from the fact that $T$ is a compact map from $\hdotzig{\balln}$ to $\hdotzig{\balln}$, and (iii) is trivial. Hence,  it remains to prove (ii). To this end,  let $\bv \in \hdotzig{\balln}$ such that 
\[ T_{\lambda}(\bv)=\bv \quad \text{for some} \quad \lambda \in [0,1].
\] 
Then, 
\begin{align*}
     -\lambda T(\bv)+\lambda \bg=\bv
\end{align*}
and therefore
\[
    -\lambda\langle T (\bv),\bv\rangle_\hdotzig{\balln}+ \lambda\langle\bg,\bv\rangle_\hdotzig{\balln}=\langle\bv,\bv \rangle_\hdotzig{\balln}.
\]
From this, it follows from the definition of $T$, $\bg$ and by \eqref{T-def}  that 
\begin{align*}
    \|\nabla\bv\|^2_\ltwo{\balln}+\lambda\langle\bv\cdot\nabla\bv,\bv\rangle_\ltwo{\balln}=\lambda\langle\bF,\nabla\bv\rangle_\ltwo{\balln}.
\end{align*}
As $\langle\bv\cdot\nabla\bv,\bv\rangle_\ltwo{\balln} =0$  by the integration by parts, we infer that 
\begin{align}  \label{energy-equality-lambda}
   & \|\nabla\bv\|^2_\ltwo{\balln} =\lambda\langle\bF,\nabla\bv\rangle_\ltwo{\balln} \\ \notag
    & \leq\|\bF\|_\ltwo{\balln}\|\nabla\bv\|_\ltwo{\balln} = \|\bF\|_{L^2(B_n^+)}\|\bv\|_{\hdotzig{\balln}}.
\end{align}
Hence, $\|\bv\|_\hdotzig{\balln}=\|\nabla\bv\|_\ltwo{\balln}\leq\|\bF\|_{L^2(B_n^+)}=M$.

\smallskip
By the claims (i)-(ii)-(iii), we apply the Leray-Schauder fixed point theorem to infer that there is $\bu_n \in \hdotzig{\balln}$ such that
\begin{equation*}
    T_{1}(\bu_n)=-(T(\bu_n)-\bg)=\bu_n.
\end{equation*}
This implies that
\begin{align*}
    &\langle\nabla\bu_n,\nabla\bphi\rangle_\ltwo{\balln}+\langle\bu_n\cdot\nabla\bu_n,\bphi\rangle_\ltwo{\balln}=\langle\bF,\nabla\bphi\rangle_\ltwo{\balln},    
\end{align*}
for all $\varphi \in C_{c,\sigma}^\infty(B_n^+)$.
Moreover, we see that \eqref{energy-equality} holds by setting $\lambda=1$ in \eqref{energy-equality-lambda} with $\bv$ replaced by $\bu_n$. The proof is completed.
\end{proof}
\begin{proof}[\textbf{Proof of  Theorem \ref{theorem2.1}}] Observe that by Lemma \ref{theorem4.1}, there exists a weak solution $\bu_n\in\hdotzig{\balln}$ to the stationary Navier-Stokes equations \eqref{1.1-ball} satisfying
\begin{equation} \label{enr-eqn-0321}
    \|\nabla\bu_n\|_\ltwo{\balln}^2 = \langle\nabla\bu,\bF\rangle_\ltwo{ \balln}, \quad \forall n \in \mathbb{N}.
\end{equation}
From this, and H\"{o}lder's inequality, we obtain
\[
\| \nabla \bu_n\|_\ltwo{B_n^+} \leq \|\bF\|_\ltwo{ \balln}, \quad \forall n \in \mathbb{N}.
\]
Then, extending  $\bu_n$ to $\upper$  by setting $\bu_n=0$ in $\upper\setminus\balln$, we see  that $\bu_n\in\hdotzig{\upper}$ and 
\begin{equation} \label{bound-u-320}
\|\nabla\bu_n\|_{L^2(\R^2_+)} \leq \|\bF\|_\ltwo{\R^2_+}, \quad \forall \ n \in \mathbb{N}.
\end{equation}
Due to Lemma \ref{Hardy}, we note that  $\hdotzig{\upper}$ is the Hilbert space with the inner product
\[
\langle\bPhi,\bphi\rangle_\hdotzig{\upper} = \int_{\R^2_+} \langle\nabla\bPhi (\bx),\nabla\bphi (\bx) \rangle \, d\bx , \quad \forall \ \bPhi, \bphi \in \hdotzig{\upper}.
\]
Hence, there exists a subsequence $(\bu_{n_k})_{k\in\N}$ of $(\bu_n)_{n\in\N}$ such that 
\begin{equation} \label{weak-con-0320}
\bu_{n_k}\rightharpoonup\bu \quad \text{in}  \quad \hdotzig{\upper} \quad \text{as} \quad  k\rightarrow\infty.
\end{equation}
 Observe that by \eqref{enr-eqn-0321} and the weak convergence \eqref{weak-con-0320}, we have
\begin{align*}
    \|\nabla\bu\|_\ltwo{\upper}^2&\leq\liminf_{k\rightarrow\infty}\|\nabla\bu_{n_k}\|_\ltwo{\upper}^2 =\langle\nabla\bu,\bF\rangle_\ltwo{\R^2_+}.
\end{align*}
Then, for $\omega \in \{\omega_0, \omega_1, \omega_2\}$, it follows from Lemma \ref{Hardy} that
\[
 \|\bu\|_\ltwo{\upper, \omega^{-1}} \leq 4 \|\nabla \bu\|_{L^2(\R^2_+)} \leq 4\|\bF\|_\ltwo{\upper}.
\]
Thus, \eqref{2.1} and \eqref{est-0321} hold, and it remains to show is that  $\bu$ is a weak solution to the stationary Navier-Stokes equations \eqref{1.1} in $\upper$, or equivalently,   
\begin{equation} \label{wea-def-u-0320}
\langle\nabla\bu,\nabla\bphi\rangle_\ltwo{\R^2_+}+\langle\bu\cdot\nabla\bu,\bphi\rangle_\ltwo{\R^2_+}=\langle\bF,\nabla\bphi\rangle_\ltwo{\R^2_+}
\end{equation}
for all $\bphi\in\comsig{\upper}$.  

Consider a fixed $\bphi\in\comsig{\upper}$, and let $m\in\N$ be sufficiently large such that the support of $\bphi$ is contained in $\ballm$. 
Observe that for $k \geq m$, as $\bu_{n_k}$ is a weak solution to  \eqref{1.1-ball} in $B_{n_k}^+$ and $\bphi \in C_{c,\sigma}^\infty(B_m^+)$, we see that
\begin{equation} \label{n-k.eqn}
    \langle\nabla\bu_{n_k},\nabla\bphi\rangle_\ltwo{\ballm}+\langle\bu_{n_k}\cdot\nabla\bu_{n_k},\bphi\rangle_\ltwo{\ballm}=\langle\bF,\nabla\bphi\rangle_\ltwo{\ballm}.
\end{equation}
Now, by \eqref{bound-u-320}, the condition $\bu_{n_k} =0 $ on $B_{n_k} \cap \{x_2 =0\}$, and Poincar\'{e}'s inequality, we see 
\[
\| \bu_{n_k}\|_{H^1(B_m^+)} \leq N(m) \|\bF\|_{L^2(\R^2_+)}, \quad \forall \ k \geq m.
\]
Then, by using \eqref{weak-con-0320} and the fact that $H^1(B_m^+)$ is compactly embedded in $\lfour{\ballm}$, there exists a subsequence of $(\bu_{n_k})_{k\in\N}$, which we still denote $(\bu_{n_k})_{k\in\N}$, such that 
\begin{equation} \label{L4-con} 
\bu_{n_k}\rightarrow\bu \quad \text{in} \quad \lfour{\ballm}  \quad \text{as} \quad k\rightarrow\infty. 
\end{equation}
Using the weak convergence \eqref{weak-con-0320} , we observe that
\[ 
    \lim_{k\rightarrow\infty}\langle\nabla\bu_{n_k},\nabla\bphi\rangle_\ltwo{\ballm} =\langle\bu,\bphi\rangle_\ltwo{\ballm}.
\]
On the other hand, for the second term on the left hand side of \eqref{n-k.eqn},  we note that 
\begin{align*}
    &\left|\langle\bu_{n_k}\cdot\nabla\bu_{n_k},\bphi\rangle_\ltwo{\ballm}-\langle\bu\cdot\nabla\bu,\bphi\rangle_\ltwo{\ballm}\right|\\
    &\leq\left|\langle(\bu_{n_k}-\bu)\cdot\nabla\bu_{n_k},\bphi\rangle_\ltwo{\ballm}\right|+\left|\langle\bu\cdot\nabla(\bu_{n_k}-\bu),\bphi\rangle_\ltwo{\ballm}\right|\\
    &\leq\left|\langle(\bu_{n_k}-\bu)\cdot\nabla\bu_{n_k},\bphi\rangle_\ltwo{\ballm}\right|+\left|\langle\bu\cdot\nabla\bphi,(\bu_{n_k}-\bu)\rangle_\ltwo{\ballm}\right|\\
    &\leq\|(\bu_{n_k}-\bu)\bphi\|_\ltwo{\ballm}\|\nabla\bu_{n_k}\|_\ltwo{\ballm} \\
    &\qquad +\|\bu(\bu_{n_k}-\bu)\|_\ltwo{\ballm}\|\nabla\bphi\|_\ltwo{\ballm}\\
    &\leq\|\bu_{n_k}-\bu\|_\lfour{\ballm}\|\bphi\|_\lfour{\ballm}\|\nabla\bu_{n_k}\|_\ltwo{\ballm}\\
    &\quad+\|\bu\|_\lfour{\ballm}\|\bu_{n_k}-\bu\|_\lfour{\ballm}\|\nabla\bphi\|_\ltwo{\ballm}\\
    &=\left(\|\bphi\|_\lfour{\ballm}\|\bF\|_\ltwo{\R^2_+}+\|\bu\|_\lfour{\ballm}\|\nabla\bphi\|_\ltwo{\ballm}\right)\|\bu_{n_k}-\bu\|_\lfour{\ballm},
\end{align*}
where we have used \eqref{bound-u-320} in the last step of the above estimate. Then, from \eqref{L4-con}, we conclude that 
\begin{align*}
    \lim_{k\rightarrow\infty}\langle\bu_{n_k}\cdot\nabla\bu_{n_k},\bphi\rangle_\ltwo{\ballm}=\langle\bu\cdot\nabla\bu,\bphi\rangle_\ltwo{\ballm}.
\end{align*}
Thus, by passing the limit as $k \rightarrow \infty$, we conclude from \eqref{n-k.eqn} that 
\begin{align*}
    \langle\nabla\bu,\nabla\bphi\rangle_\ltwo{\ballm}+\langle\bu\cdot\nabla\bu,\bphi\rangle_\ltwo{\ballm}=\langle\bF,\nabla\bphi\rangle_\ltwo{\ballm},
\end{align*}
which implies \eqref{wea-def-u-0320}, and the proof is completed.
\end{proof}

\subsection{Weak-strong uniqueness of solutions}
This subsection proves Theorem \ref{theorem2.2}.  We begin with the following lemma which is needed for the proof.
\begin{lemma} \label{lemma5.4} Assume that  $\omega \in \{\omega_0, \omega_1, \omega_2\}$. Then, the following statements hold true.
\begin{itemize}
\item[\textup{(i)}] If $\bv\in\hdotzig{\upper}\cap[\linf{\R^2_+,\omega^{1/2}}\cup\mathcal X]$ and $\bu\in\hdotzig{\upper}$, then 
\[
    \langle\bu\cdot\nabla\bv,\bv\rangle_\ltwo{\upper}=0.   
\]
\item[\textup{(ii)}] If $\bu\in\hdotzig{\upper}$ is a weak solution to the stationary Navier-Stokes equations \eqref{1.1} in $\upper$, then 
	\[
\langle\nabla\bu,\nabla\bv\rangle_\ltwo{\upper}+\langle\bu\cdot\nabla\bu,\bv\rangle_\ltwo{\upper}=\langle\bF,\nabla\bv\rangle_\ltwo{\upper}.
	\]
	for every $\bv\in\hdotzig{\upper}\cap[\linf{\R^2_+,\omega^{1/2}}\cup\mathcal X]$.
	\end{itemize}
\end{lemma}
\begin{proof} We start with proving (i). We first observe that as $\bv\in\hdotzig{\upper}\cap[\linf{\R^2_+,\omega^{1/2}}\cup\mathcal X]$, it follows from H\"{o}lder's inequality, \eqref{uv-finite-linfty}, and \eqref{uv-finite-chi} that
\[
\begin{split}
|\langle\bu\cdot\nabla\bv,\bv\rangle_\ltwo{\upper}| \leq \|\nabla\bv\|_{L^2{(\R^2_+)}} \|\bu\otimes\bv\|_\ltwo{\R^2_+}
 <\infty.
\end{split}
\]
Now, according to Lemma \ref{approximationlemma} and Lemma \ref{approximationlemma2}, there exists $(\bv_n)_{n\in\N}\subset\comsig{\upper}$ such that
\begin{equation} \label{apprlemmaresult}
	\|\nabla (\bv-\bv_n) \|_{L^2(\R^2_+)} \rightarrow 0 \quad \text{and} \quad  \|\bu\otimes(\bv-\bv_n)\|_\ltwo{\upper} \rightarrow 0
\end{equation}
as $n \rightarrow \infty$. We integrate by parts, which results in
\begin{equation} \label{5.4}
     \langle\bu\cdot\nabla\bv,\bv_n\rangle_\ltwo{\upper} =-\langle\bu\cdot\nabla\bv_n,\bv\rangle_\ltwo{\upper}.
\end{equation}
Then, with H\"{o}lder's inequality, it follows that
\begin{align*}
    \left|\langle\bu\cdot\nabla\bv,\bv-\bv_n\rangle_\ltwo{\upper}\right|\leq\|\nabla\bv\|_\ltwo{\upper}\|\bu\otimes(\bv-\bv_n)\|_\ltwo{\upper}
\end{align*}
and 
\begin{align*}
    \left|\langle\bu\cdot\nabla(\bv-\bv_n),\bv\rangle_\ltwo{\upper}\right|&\leq\|\nabla(\bv_n-\bv)\|_\ltwo{\upper}\|\bu\otimes\bv\|_\ltwo{\upper}\\
   &=\|\bv_n-\bv\|_\hdotzig{\upper}\|\bu\otimes\bv\|_\ltwo{\upper}.
\end{align*}
Then, by using \eqref{apprlemmaresult}, we pass the limit as $n \rightarrow \infty$ of \eqref{5.4}  to get
\[
 \langle\bu\cdot\nabla\bv,\bv\rangle_\ltwo{\upper} =- \langle\bu\cdot\nabla\bv,\bv\rangle_\ltwo{\upper} .
\]
Hence, $\langle\bu\cdot\nabla\bv,\bv\rangle_\ltwo{\upper}  =0$ and the assertion (i) is proved.

Next, we prove (ii). Let $\bu\in\hdotzig{\upper}$ be a weak solution to the Navier-Stokes equations \eqref{1.1} in $\upper$, and let $\bv\in\hdotzig{\upper}\cap[\linf{\R^2_+,\omega^{1/2}}\cup\mathcal X]$. As in the proof of (i), let $(\bv_n)_{n\in\N}\subset\comsig{\upper}$ satisfy \eqref{apprlemmaresult}, whose existence is ensured by Lemmas \ref{approximationlemma}, \ref{approximationlemma2}. Then, because $\bu$ is a weak solution to \eqref{1.1}, we have 
	\begin{align*}
		\langle\nabla\bu,\nabla\bv_n\rangle_\ltwo{\upper}+\langle\bu\cdot\nabla\bu,\bv_n\rangle_\ltwo{\upper}=\langle\bF,\nabla\bv_n\rangle_\ltwo{\upper},
	\end{align*}
for all $n \in \mathbb{N}$. Since $\bv_n\rightarrow\bv$ in $\hdotzig{\upper}$ as in  \eqref{apprlemmaresult} when $n\rightarrow\infty$, we have
\[
\lim_{n\rightarrow \infty} \langle\nabla\bu,\nabla\bv_n\rangle_\ltwo{\upper} = \langle\nabla\bu,\nabla\bv\rangle_\ltwo{\upper}
\]
and
\[
\lim_{n\rightarrow \infty}\langle\bF,\nabla\bv_n\rangle_\ltwo{\upper} = 
\langle\bF,\nabla\bv\rangle_\ltwo{\upper}.
\]	
On the other hand, we also have
	\begin{align*}
		&  \left|\langle\bu\cdot\nabla\bu,\bv\rangle_\ltwo{\upper}-\langle\bu\cdot\nabla\bu,\bv_n\rangle_\ltwo{\upper}\right|\\
		&=\left|\langle\bu\cdot\nabla\bu,(\bv-\bv_n)\rangle_\ltwo{\upper}\right|\\
		  &\leq\|\nabla\bu\|_\ltwo{\upper}\|\bu\otimes(\bv-\bv_n)\|_\ltwo{\upper} \rightarrow 0
	\end{align*}
 as $n\rightarrow\infty$ by \eqref{apprlemmaresult}. Therefore, the assertion (ii) follows, and the proof of the lemma is completed.
\end{proof}
 
The next lemma will also be needed for our proof of the weak-strong uniqueness of solutions.
\begin{lemma} \label{lemma5.3} Let  $\bu \in\hdotzig{\R^2_+}$ be a weak solution to the stationary Navier-Stokes equations \eqref{1.1}. Assume that $\bu \in \linf{\R^2_+,\omega^{1/2}}\cup\mathcal X$ for some $\omega \in \{\omega_0, \omega_1, \omega_2\}$. Then
\[
    \langle\nabla \bu,\nabla\bphi\rangle_\ltwo{\upper}-\langle \bu \cdot\nabla\bphi, \bu\rangle_\ltwo{\upper}=\langle\bF,\nabla\bphi\rangle_\ltwo{\upper}
\]
for all $\bphi \in \hdotzig{\R^2_+}$.
\end{lemma}
\begin{proof} As $\bphi \in\hdotzig{\R^2_+}$, we can find  a sequence $(\bphi_n)_{n\in\N}\subset\comsig{\upper}$ such that
\begin{equation} \label{eqn:3-20}
\lim_{n\rightarrow \infty}\|\nabla \bphi_n - \nabla \bphi\|_{L^2(\R^2_+)} =0. 
\end{equation}
Since $\bu \in \hdotzig{\R^2_+}$ is a weak solution to \eqref{1.1}, we have 
\[
\langle\nabla \bu,\nabla\bphi_n\rangle_\ltwo{\upper}+\langle \bu \cdot\nabla \bu,\bphi_n\rangle_\ltwo{\upper} =\langle\bF,\nabla\bphi_n\rangle_\ltwo{\upper}.
\]
Then, we integrate by parts to give 
\begin{equation} \label{5.3-1}
    \langle\nabla \bu,\nabla\bphi_n\rangle_\ltwo{\upper}-\langle \bu\cdot\nabla\bphi_n, \bu\rangle_{L^2(\R^2_+)}=\langle\bF,\nabla\bphi_n\rangle_\ltwo{\upper}.
\end{equation}
It follows directly from \eqref{eqn:3-20} that
\[
\begin{split}
& \lim_{n\rightarrow \infty} \langle\nabla \bu,\nabla\bphi_n\rangle_\ltwo{\upper} = \langle\nabla \bu,\nabla\bphi\rangle_\ltwo{\upper} \quad \text{and} \\
& \lim_{n\rightarrow \infty} \langle\bF,\nabla\bphi_n\rangle_\ltwo{\upper} = \langle\bF,\nabla\bphi\rangle_\ltwo{\upper}.
\end{split}
\]
On the other hand, for each $n \in \mathbb{N}$, it follows from H\"{o}lder's inequality, \eqref{uv-finite-linfty}, and \eqref{uv-finite-chi} that
\begin{align*}
 &   \left|\langle \bu\cdot(\nabla\bphi_n-\nabla\bphi), \bu\rangle_\ltwo{\upper}\right| \\
    &\leq\|\nabla(\bphi_n-\bphi)\|_\ltwo{\upper}\| \bu\otimes \bu\|_\ltwo{\upper}\\
        &\leq N\|\nabla(\bphi_n-\bphi)\|_\ltwo{\upper}.
\end{align*}
Then, from \eqref{eqn:3-20}, we infer that
\[
\lim_{n\rightarrow \infty} \langle \bu \cdot \nabla\bphi_n, \bu\rangle_\ltwo{\upper}  = \langle \bu \cdot \nabla\bphi, \bu\rangle_\ltwo{\upper}.
\]
Hence, by taking the limit as $n\rightarrow \infty$ on both sides of \eqref{5.3-1}, we obtain the desired result.   The proof of the lemma is completed.
\end{proof}
We are now ready to prove Theorem \ref{theorem2.2}.
\begin{proof}[\textbf{Proof of  Theorem \ref{theorem2.2}}] We consider the assumption (i). Let us denote $\omega = \omega_i$ with some fixed $i =0,1,2$ so that $\bar{\bu} \in L^\infty(\R^2_+, \omega^{1/2})$ and
\begin{equation} \label{2.2-bis}
    \|\bar{\bu}\|_{L^\infty(\R^2_+, \omega^{1/2})}< 1/4.
\end{equation}
As $\bar{\bu} \in L^\infty(\R^2_+, \omega^{1/2})$, we apply (ii) of Lemma \ref{lemma5.4} to infer that
\begin{equation} \label{0321-1}
    \langle\nabla\bu,\nabla \bar{\bu}\rangle_\ltwo{\upper}+\langle\bu\cdot\nabla\bu,\bar{\bu}\rangle_\ltwo{\upper}=\langle\bF,\nabla\bar\bu\rangle_\ltwo{\upper}.
\end{equation}
Similarly, we apply Lemma \ref{lemma5.3} to get
\begin{equation} \label{0321-2}
    \langle\nabla\bar{\bu},\nabla\bu\rangle_\ltwo{\upper}-\langle\bar{\bu} \cdot\nabla\bu,\bar{\bu}\rangle_\ltwo{\upper}=\langle\bF,\nabla\bu\rangle_\ltwo{\upper}.
\end{equation}
Next, let us denote $\bw =\bu-\bar{\bu}$.  By a simple algebra, we have
\[
\begin{split}
\|\nabla\bw\|^2_\ltwo{\upper}  & =\|\nabla\bu\|^2_\ltwo{\upper}+\|\nabla\bar{\bu}\|^2_\ltwo{\upper} \\
& \qquad -\langle\nabla\bu,\nabla\bar{\bu}\rangle_\ltwo{\upper}  -\langle\nabla\bar{\bu},\nabla\bu\rangle_\ltwo{\upper}.
\end{split}
\]
From this, \eqref{0321-1}, and \eqref{0321-2} 
it follows that
\begin{align*}
    \|\nabla\bw\|^2_\ltwo{\upper} 
    &=\|\nabla\bu\|^2_\ltwo{\upper}+\|\nabla\bar{\bu}\|^2_\ltwo{\upper}+\langle\bu\cdot\nabla\bu,\bar{\bu}\rangle_\ltwo{\upper}\\
    & \qquad -\langle\bF,\nabla\bar\bu\rangle_\ltwo{\upper}-\langle\bar\bu\cdot\nabla\bu,\bar{\bu}\rangle_\ltwo{\upper}-\langle\bF,\nabla\bu\rangle_\ltwo{\upper}\\
    &=\|\nabla\bu\|^2_\ltwo{\upper}+\|\nabla\bar\bu\|^2_\ltwo{\upper}-\langle\bF,\nabla\bar\bu\rangle_\ltwo{\upper}
    \\
    & \qquad -\langle\bF,\nabla\bu\rangle_\ltwo{\upper} +\langle\bw\cdot\nabla\bu,\bar{\bu}\rangle_\ltwo{\upper},
\end{align*}
which together with \eqref{2.1} imply that
\[
 \|\nabla\bw\|^2_\ltwo{\upper} \leq\langle\bw\cdot\nabla\bu,\bar{\bu}\rangle_\ltwo{\upper}.
\]
Also, due to (i) of Lemma \ref{lemma5.4}, we see that $\langle\bw\cdot\nabla\bar{\bu},\bar{\bu}\rangle_\ltwo{\upper} =0$, and therefore
\begin{align*}
    \|\nabla\bw\|^2_\ltwo{\upper} & \leq \langle\bw\cdot\nabla\bu,\bar{\bu}\rangle_\ltwo{\upper}-\langle\bw\cdot\nabla\bar{\bu},\bar{\bu}\rangle_\ltwo{\upper}\\&=\langle\bw\cdot\nabla\bw,\bar{\bu}\rangle_\ltwo{\upper} \leq\|\nabla\bw\|_\ltwo{\upper}\|\bw\bar{\bu}\|_\ltwo{\upper}.
\end{align*}
Finally, we infer from this last estimate, Lemma \ref{Hardy}, and \eqref{2.2-bis} that
\begin{align*} 
 \|\nabla\bw\|_\ltwo{\upper}& \leq\|\bw\bar{\bu}\|_\ltwo{\upper}\leq\|\bw\|_{\ltwo{\upper, \omega^{-1}}}\|\bar{\bu}\|_{\linf{\upper, \omega^{1/2}}} \\
 & \leq 4\|\nabla\bw\|_\ltwo{\upper} \|\bar{\bu}\|_{\linf{\upper, \omega^{1/2}}} \\
 & < \|\nabla\bw\|_\ltwo{\upper}.
\end{align*}
This implies $\|\nabla\bw\|_\ltwo{\upper}=0$. From this, and the boundary condition $\bw|_{\partial\upper}=0$, we get $\bw\equiv0$ in $\upper$. Therefore, $\bu = \bar{\bu}$.
\smallskip

 Now we consider the assumption (ii) of Theorem \ref{theorem2.2}. We have by a similar argument as in the previous case that
\[
 \|\nabla\bw\|_\ltwo{\upper} \leq\|\bw\bar{\bu}\|_\ltwo{\upper}.
\]
Then, using \eqref{3.3-2} from Lemma \ref{Hardy-2} for $\bw$, we obtain
\[
\begin{split}
\|\bw\bar{\bu}\|_\ltwo{\upper} &= \left(\int_0^\infty \int_{\R}|\bw(\bx)|^2|\bar{\bu}(\bx)|^2 dx_1 dx_2\right)^{1/2} \\
& \leq \left(\int_0^\infty\Big[ \|\bar{\bu}(\cdot, x_2)\|_{L^\infty(\R)}^2 \int_{\R} | \bw(\bx)|^2 dx_1\Big] dx_2 \right)^{1/2} \\
& \leq \|\nabla \bw\|_{L^2(\R^2_+)} \left(\int_0^\infty x_2\|\bar{\bu}(\cdot, x_2)\|_{L^\infty(\R)}^2 dx_2 \right)^{1/2}.
\end{split}
\]
Then, by the assumption
\[
\left (\int_0^\infty x_2\|\bar{\bu}(\cdot, x_2)\|_{L^\infty(\R)}^2 dx_2 \right)^{1/2}  < 1,
\]
we obtain
\[
\|\nabla\bw\|_\ltwo{\upper}  < \|\nabla\bw\|_\ltwo{\upper}, 
\]
which implies $\bw =0$ and then $\bu = \bar{\bu}$. The proof is completed.
\end{proof}
\section*{Acknowledgement}
This work was done partially when Adrian Calderon was an undergraduate student at the University of Tennessee. Adrian Calderon would like to thank the University of Tennessee for the generous supports.

\end{document}